\documentclass[12pt]{article}
\usepackage{framed}
\RequirePackage{kpfonts}
\RequirePackage[sf,bf,small,raggedright]{titlesec}
\usepackage{textcomp}
\usepackage{amssymb}
\usepackage{bbm}
\usepackage{fancyhdr}
\usepackage{amsmath}
\usepackage{amsbsy,amsthm}
\usepackage{amscd}
\usepackage{latexsym}
\usepackage{graphicx}   
\usepackage{pdfsync}
\usepackage{blkarray}
\usepackage{multirow}
\usepackage{hyperref}
\usepackage{tcolorbox}

\newcommand{\R}{\mathbb{R}}
\newcommand{\Rd}{\R^d}
\newcommand{\inr}[1]{\left\langle #1 \right\rangle}

\newcommand{\E}{\mathbb{E}}

\newcommand{\cD}{\mathcal{D}}

\newcommand{\eps}{\varepsilon}

\newcommand{\radius}{{\rm radius}}

\newtheorem{lemma}{Lemma}
\newtheorem{theorem}{Theorem}

\newtheorem{proposition}{Proposition}

\newtheorem{remark}{remark}

\numberwithin{equation}{section}

\def \endproof
{{\mbox{}\nolinebreak\hfill\rule{2mm}{2mm}\par\medbreak}}
\def\IND{\mathbbm{1}}

\newcommand{\ol}{\overline}

\newcommand{\wh}{\widehat}
\newcommand{\argmin}{\mathop{\mathrm{argmin}}}
\newcommand{\X}{\mathcal{X}}

\newcommand{\EXP}{\mathbb{E}}
\newcommand{\PROB}{\mathbb{P}}

\newcommand{\var}{\mathrm{Var}}
\newcommand{\F}{{\mathcal F}}

\newcommand{\Tr}{\mathrm{Tr}}
\newcommand{\lambdamax}{\lambda_{\text{max}}}

\newcommand{\defeq}{\stackrel{\mathrm{def.}}{=}}

\def\Bin{{\mathrm {Bin}}}

\begin{document}

\title{Mean estimation and regression under heavy-tailed
  distributions---a survey
\thanks{
G\'abor Lugosi was supported by
the Spanish Ministry of Economy and Competitiveness,
Grant MTM2015-67304-P and FEDER, EU, by
``High-dimensional problems in structured probabilistic models - Ayudas Fundaci\'on BBVA a Equipos de Investigaci\'on Cientifica 2017'' and by ``Google Focused Award Algorithms and Learning for AI''.
 Shahar Mendelson was supported in part by the Israel Science Foundation.
}
}
\author{
G\'abor Lugosi\thanks{Department of Economics and Business, Pompeu
  Fabra University, Barcelona, Spain, gabor.lugosi@upf.edu}
\thanks{ICREA, Pg. Lluís Companys 23, 08010 Barcelona, Spain}
\thanks{Barcelona Graduate School of Economics}
\and
Shahar Mendelson \thanks{Mathematical Sciences Institute, The
  Australian National University and LPSM, Sorbonne University, shahar.mendelson@anu.edu.au}}

\maketitle

\begin{abstract}
We survey some of the recent advances in mean estimation and
regression function estimation. In particular, we describe 
sub-Gaussian mean estimators
for possibly heavy-tailed data both in the univariate and multivariate
settings. We focus on estimators based on median-of-means techniques
but other methods such as the trimmed mean and Catoni's estimator are also reviewed. 
We give detailed proofs for the cornerstone results. We dedicate a
section on statistical learning problems--in particular, regression
function estimation--in the presence of possibly heavy-tailed data.
\end{abstract}




\newpage
\tableofcontents

\section{Introduction}

Arguably the most fundamental problem of statistics is that of
estimating the expected value $\mu$ of a random variable $X$ based on
a sample of $n$ independent, identically distributed draws from the
distribution of $X$. The obvious choice of an estimator is, of course,
the empirical mean. Its properties are well understood by classical 
results of probability theory. However, from the early days on,
statisticians have been concerned about the quality of the empirical
mean, especially when the distribution may be \emph{heavy-tailed} or
\emph{outliers} may be present in the data. This concern gave rise to
the area of \emph{robust statistics} that addresses the problem of
mean estimation (and other statistical problems) for such data. 
Classical references include
Huber \cite{Hub64},
Huber and Ronchetti \cite{HuRo09},
Hampel, Ronchetti, Rousseeuw, and Stahel \cite{HaRoRoSt86},
Tukey \cite{Tuk75}. 

Motivated by applications in machine learning and
data science, in recent years there has been increased interest
in constructing mean and regression function estimates with the
requirement that the estimators should achieve  \emph{high accuracy}
with a large \emph{confidence}. The best achievable
accuracy/confidence tradeoff is much better understood today and the
aim of this paper is to survey some of the recent advances. We
primarily focus on the mean estimation problem, both in the univariate 
and multivariate settings. We offer detailed discussion of what the
best performance one may expect is, describe a variety of
estimators, and analyze their performance. We pay special attention to a simple
but powerful methodology based on \emph{median-of-means} techniques.

We also address one of the basic problems of statistical learning
theory, namely regression function estimation. We show how the 
technology introduced for mean estimation may be used to construct
powerful learning algorithms that achieve essentially optimal
performance under mild assumptions. 

The paper is organized as follows. In Section \ref{sec:univariate} we 
address the simplest, univariate mean estimation problem. We focus on
\emph{sub-Gaussian} estimators and explore their possibilities and
limitations. Section \ref{sec:vectorest} is dedicated to the
significantly more challenging multivariate problem. We extend the
notion of sub-Gaussian estimators to the multivariate setting and
analyze various estimators. In Section \ref{sec:gennorm} we study
the problem of estimating the mean of an entire class of random
variables with the requirement that all estimators have a high
accuracy simultaneously over the entire class. We show how such
estimators may be constructed and use these ideas in a general 
framework of mean estimation. Finally, Section \ref{sec:regression}
is dedicated to applying these techniques to regression function estimation.

\section{Estimating the mean of a real random variable}
\label{sec:univariate}

In this section we examine the classical problem of estimating the mean of a
random variable. Let $X_1,\ldots,X_n$ be independent, identically distributed
real random variables with mean $\mu=\EXP X_1$. Upon observing these random
variables, one would like to estimate $\mu$. An estimator $\wh{\mu}_n=\wh{\mu}_n(X_1,\ldots,X_n)$ is simply
a measurable function of $X_1,\ldots,X_n$.

The quality of an estimator may be measured in various ways.
While most of the early statistical work focused on expected risk measures such as the
\emph{mean-squared error}
\[
   \EXP\left[ \left( \wh{\mu}_n - \mu\right)^2 \right]~,
\]
such risk measures may be misleading. Indeed, if the difference $|\wh{\mu}_n - \mu|$
is not sufficiently concentrated, the expected value does not necessarily
reflect the ``typical'' behavior of the error. For such reasons,
we prefer estimators $\wh{\mu}_n$ that are close to $\mu$ \emph{with high probability}.
Thus, our aim is to understand, for any given sample size $n$ and confidence parameter $\delta \in (0,1)$,
 the smallest possible value $\epsilon=\epsilon(n,\delta)$ such that
\begin{equation} \label{eq:PAC-intro}
   \PROB\left\{ \left| \wh{\mu}_n - \mu \right| > \epsilon \right\} \le \delta~.
\end{equation}
It it important to stress that \eqref{eq:PAC-intro} is a
non-asymptotic criterion: one would like to obtain quantitative
estimates on the way the accuracy $\epsilon$ scales with the
confidence parameter $\delta$ and the sample size $n$. This type of
estimate is reminiscent to the \textsc{pac} (Probably Approximately
Correct) framework usually adopted in statistical learning theory, see
Valiant \cite{Val84}, Vapnik and Chervonenkis \cite{VaCh74a},  Blumer,
Ehrenfeucht, Haussler, and Warmuth \cite{BlEhHaWa89}.

The most natural choice of a mean estimator is the standard empirical mean
$$\ol\mu_n=\frac{1}{n}\sum_{i=1}^nX_i~.$$
The behavior of the empirical mean is well understood. For example, if
the $X_i$ have a finite second moment and $\sigma^2$ denotes their variance, then
the mean-squared error of $\ol\mu_n$ equals $\sigma^2/n$. On the other hand,
the central limit theorem guarantees that
this estimator has Gaussian tails, asymptotically, when $n\to \infty$. Indeed,
\[
\PROB\left\{
\left|\ol\mu_n - \mu \right|>\frac{\sigma \Phi^{-1}(1-\delta/2)}{\sqrt{n}}\right\}\to \delta~,
\]
where $\Phi(x)=\PROB\{G\le x\}$ is the cumulative distribution function of a standard normal random variable $G$.
One may easily see (e.g., using the fact that for $t \geq 1$, $\exp(-t^2/2) \leq t\exp(-t^2/2)$ ), that for all $x\ge 0$,
\[
1-\Phi(x) \le e^{-x^2/2}~.
\]
This implies that
$\Phi^{-1}(1-\delta/2)\le \sqrt{2\log(2/\delta)}$, and the central limit theorem asserts that
\[
\lim_{n\to\infty}
\PROB\left\{
\left|\ol\mu_n - \mu \right|>\frac{\sigma  \sqrt{2\log(2/\delta)}}{\sqrt{n}}\right\}\le \delta~.
\]
However, this is an asymptotic estimate and not the quantitative one we were hoping for. Still, our goal is to obtain non-asymptotic performance bounds of the same form. In particular, we say that a mean estimator $\wh{\mu}_n$ is $L$-\emph{sub-Gaussian} if there is a constant $L>0$, such that for all sample sizes $n$ and with probability at least $1-\delta$,
\begin{equation}
\label{eq:subgauss}
\left|\wh\mu_n - \mu \right| \le \frac{L\sigma  \sqrt{\log(2/\delta)}}{\sqrt{n}}~.
\end{equation}

It is worth noting here the well-known fact that if all one knows is
that the unknown distribution is Gaussian, then the sample mean is
optimal for all sample sizes and confidence levels $\delta$. (See
Catoni \cite[Proposition 6.1]{Cat10} for a precise statement.)
Moreover, the following observation, established by Devroye, Lerasle,
Lugosi, and Oliveira \cite{DeLeLuOl16}, shows that \eqref{eq:subgauss}
is essentially the best that one can hope for in general, even if one
is interested in a fixed confidence level:
\begin{theorem} \label{thm:infvar}
Let $n>5$ be a positive integer. Let $\mu\in \R$, $\sigma>0$ and $\delta
\in (2e^{-n/4},1/2)$. Then for any mean estimator $\wh\mu_n$, there exists a distribution
with mean $\mu$ and variance $\sigma^2$ such that
\[
\PROB\left\{ \left|\wh\mu_n - \mu\right|> \sigma   \sqrt{\frac{\log(1/\delta)}{n}} \right\} \geq \delta~.
\]
\end{theorem}

\begin{proof}
To derive the ``minimax'' lower bound, it suffices to consider two
distributions,
$P_+,P_-$, both concentrated on two points, defined by
\[
   P_+(\{0\}) =   P_-(\{0\}) = 1-p~,   \qquad  P_+(\{c\}) =    P_-(\{-c\}) = p~,
\]
where $p\in [0,1]$ and $c>0$.
Note that the means of the two distributions are $\mu_{P_+}= pc$ and $\mu_{P_-}= -pc$
and both have variance $\sigma^2=c^2p(1-p)$.

For $i=1,\ldots,n$, let $(X_i,Y_i)$ be independent pairs of
real-valued random variables such that
\[
   \PROB\{X_i=Y_i=0\} = 1-p\quad \text{and} \quad \PROB\{X_i=c, Y_i=-c\} = p~.
\]
Note that $X_i$ is distributed as $P_+$ and $Y_i$ is distributed as $P_-$.
Let $\delta \in (0,1/2)$. If  $\delta \ge 2e^{-n/4}$ and  $p =
(1/(2n))\log(2/\delta)$, then (using $1-p\ge \exp(-p/(1-p))$),
\[
  \PROB\{X_1^n = Y_1^n\} =(1-p)^n \ge 2\delta~.
\]
Let $\wh\mu_n$ be any mean estimator, possibly depending on $\delta$. Then
\begin{eqnarray*}
\lefteqn{
\max\left(
\PROB\left\{ \left|\wh\mu_n (X_1^n) - \mu_{P_+}\right| >  cp\right\},
\PROB\left\{ \left|\wh\mu_n(Y_1^n) - \mu_{P_-}\right| >  cp \right\}
\right)
}
\\
& & \ge \frac{1}{2}
\PROB\left\{ \left|\wh\mu_n (X_1,\ldots,X_n) - \mu_{P_+}\right| >  cp
  \quad\text{or} \quad
 \left|\wh\mu_n(Y_1,\ldots,Y_n) - \mu_{P_-}\right| >  cp \right\}
\\
& &\ge
\frac{1}{2}\PROB\{\wh\mu_n(X_1,\ldots,X_n) = \wh\mu_n(Y_1,\ldots,Y_n)\}
\\
& &
\ge
\frac{1}{2} \PROB\{X_1,\ldots,X_n = Y_1,\ldots,Y_n\}\ge \delta~.
\end{eqnarray*}
From $\sigma^2=c^2p(1-p)$ and $p\le 1/2$ we have that $cp\ge \sigma \sqrt{p/2}$, and therefore
\begin{eqnarray*}
\max\left(
\PROB\left\{ \left|\wh\mu_n (X_1,\ldots,X_n) - \mu_{P_+}\right| >\sigma
\sqrt{\frac{\log\frac{2}{\delta}}{n}}\right\}~,
\PROB\left\{ \left|\wh\mu_n(Y_1,\ldots,Y_n) - \mu_{P_-}\right| >  \sigma
\sqrt{\frac{\log\frac{2}{\delta}}{n}}\right\}
\right) \ge \delta~.
\end{eqnarray*}
 Theorem \ref{thm:infvar} follows.
\end{proof}

With Theorem \ref{thm:infvar} in mind, our aim is to consider both
univariate and multivariate situations and design estimators that 
perform with sub-Gaussian error rate. The meaning of sub-Gaussian
error rate in the multivariate case is 
explained in Section \ref{sec:vectorest}.

Naturally, the first order of business is to check whether the obvious
choice of a mean estimator---the empirical mean---is
$L$-sub-Gaussian for some $L$. On the one hand, it is easy to see that under
certain conditions on the distribution of the $X_i$, it does exhibit a
sub-Gaussian performance. Indeed, if the $X_i$ are such that there
exists $L>0$ such that for all
$\lambda>0$
\[
   \EXP e^{\lambda(X_i-\mu)} \le e^{\sigma^2\lambda^2/L^2}~,
\]
then the empirical mean $\wh\mu_n$ is $L$-sub-Gaussian for all $\delta
\in (0,1)$, as it is easily seen by the Chernoff bound.

On the other hand, assumptions of this type are quite restrictive and
impose strong conditions on the decay of the tail probabilities of the
$X_i$. Specifically, it is equivalent to the fact that for every $p
\geq 2$, $\left(\EXP |X_i-\mu|^p\right)^{1/p} \leq L^\prime \sqrt{p} \left(\EXP |X_i-\mu|^2\right)^{1/2}$,
where $c_1 L \leq L^\prime \leq c_2 L $ for suitable absolute
constants $c_1$ and $c_2$ (see, e.g., \cite{BoLuMa13}).  
When the $X_i$'s do not exhibit such a tail decay, the empirical mean need not be sub-Gaussian.

For example, if one only assumes that $\sigma$ exists (i.e.,  the
variance of the $X_i$ is finite) then the bound implied by Chebyshev's
inequality, that is, that with probability at least $1-\delta$,
\begin{equation} \label{eq:chebyshev-intro}
  \left| \ol\mu_n - \mu \right| \le \sigma \sqrt{\frac{1}{n\delta}}~,
\end{equation}
is essentially the best that one can hope for. Although the bound from \eqref{eq:chebyshev-intro} decays with the sample size at the optimal rate of $O(n^{-1/2})$, the dependence
on the confidence parameter $\delta$ is exponentially worse than in \eqref{eq:subgauss}.
We refer to Catoni \cite[Proposition 6.2]{Cat10} for a precise
formulation and a simple example that (almost) saturates Chebyshev's inequality.

This leads to an inevitable conclusion: if one is looking for a  mean estimator that is sub-Gaussian for any random variable that has a well-defined mean and variance, then one must find alternatives to the sample mean. As it happens, and perhaps surprisingly, there exist mean estimators that achieve a sub-Gaussian performance for all distributions
with a finite variance. Two quite different estimators are presented and analyzed in the next two sections.

\subsection{The median-of-means estimator} 
\label{sec:MOM}

The median-of-means estimator presented next has been proposed in
different forms in various papers, see 
Nemirovsky and Yudin \cite{NeYu83}, Hsu \cite{HsuBlog}, Jerrum, Valiant, and Vazirani \cite{JeVaVa86}, Alon, Matias, and Szegedy \cite{AlMaSz02}.

The definition of the median-of-means estimator calls for partitioning
the data into $k$ groups of roughly equal size, computing the
empirical mean in each group, and taking the median of the obtained
values.

Formally, recall that the median of $k$ real numbers
$x_1,\ldots,x_k\in\R$ is defined as $M(x_1,\ldots,x_k) = x_i$ where $x_i$ is such that
\[
\left| \{j\in [k]\,:\,x_j\leq x_i\} \right| \geq \frac{k}{2} \quad \text{and} \quad \left| \{j\in [k]\,:\,x_j\geq x_i\}\right| \geq \frac{k}{2}~.
\]
(If several indices $i$ fit the above description, we take the smallest one.)

Now let $1\le k\le n$ and partition $[n]=\{1,\dots,n\}$ into $k$ blocks $B_1,\ldots,B_k$, each of
size
$|B_i|\geq \lfloor n/k\rfloor\geq 2$.

Given $X_1,\ldots,X_n$,
compute the sample mean in each block
$$Z_j=\frac{1}{|B_j|}\sum_{i\in B_j}X_i$$
and define the median-of-means estimator by
$\wh{\mu}_n= M(Z_1,\ldots,Z_k).$

To grasp intuitively why this estimator works, note that
  for each block, the empirical mean is an unbiased estimator of the
  mean, with controlled standard deviation $\sigma/\sqrt{n/k}$.  Hence, the
  median of the distribution of the blockwise empirical mean lies
  within $\sigma/\sqrt{n/k}$ from the expectation. Now the empirical median is a
  highly concetrated estimator of this
  median.   

A performance-bound of the estimator is established next. For simplicity, assume that
$n$ is divisible by $k$ so that each block has $m=n/k$ elements.

\begin{theorem} \label{thm:mom}
Let $X_1,\ldots,X_n$ be independent, identically distributed random variables
with mean $\mu$ and variance $\sigma^2$.
Let $m,k$ be positive integers assume that $n=mk$. Then the median-of-means
estimator $\wh\mu_n$ with $k$ blocks satisfies
\[
   \PROB\left\{  \left|\wh\mu_n-\mu \right| > \sigma \sqrt{4/m} \right\} \le e^{-k/8}~.
\]
In particular, for any $\delta\in (0,1)$, if $k= \left\lceil 8 \log(1/\delta) \right\rceil$, then, with probability
at least $1-\delta$,
\[
   \left|\wh\mu_n-\mu \right| \le  \sigma \sqrt{\frac{32\log(1/\delta)}{n}}~.
\]
\end{theorem}

\begin{proof}
By Chebyshev's inequality, for each $j=1,\ldots,k$, with probability  at least $3/4$,
\[
  \left| Z_j - \mu \right| \le \sigma \sqrt{\frac{4}{m}}~.
\]
Thus, $\left|\wh\mu_n-\mu \right| > \sigma \sqrt{4/m}$ implies that at least
$k/2$ of the means $Z_j$ are such that $\left| Z_j - \mu \right| > \sigma \sqrt{4/m}$.
Hence,
\begin{eqnarray*}
\PROB\left\{  \left|\wh\mu_n-\mu \right| > \sigma \sqrt{4/m} \right\}
& \le & \PROB\left\{ \Bin(k,1/4)\ge \frac{k}{2} \right\}   \\
& & \text{(where $\Bin(k,1/4)$ is a binomial $(k,1/4)$ random variable)} \\
& = & \PROB\left\{ \Bin(k,1/4) - \EXP \Bin(k,1/4) \ge \frac{k}{4} \right\}   \\
& \le & e^{-k/8}\quad \text{(by Hoeffding's inequality \cite{Hoe63}).}
\end{eqnarray*}
\end{proof}

Theorem \ref{thm:mom} shows that the median-of-means estimator has a
sub-Gaussian performance with $L=8$ for all distributions with a
finite variance. However, it is important to point out that the
estimator $\wh\mu_n$ depends on the confidence level $\delta$ as the
number of blocks $k$ is chosen as a function of $\delta$. This is not
a desirable property, since for different values of the confidence
parameter $\delta$, one obtains a different point estimator. However, 
as it is shown in Section \ref{sec:multidelta} below, there do not
exist sub-Gaussian estimators that are independent of the confidence
level, unless one is willing to assume more than just the finiteness
of the second moment of the underlying distribution.

The results of Bubeck, Cesa-Bianchi, and Lugosi \cite{BuCeLu13} and
Devroye, Lerasle, Lugosi, and Oliveira \cite{DeLeLuOl16} show that the
median-of-means estimator may also be used even if the distribution of the $X_i$ has an infinite variance but has a finite moment of order $1+\alpha$ for some $\alpha\in (0,1)$.
\vskip0.3cm

\begin{theorem} \label{thm:moment}
Let $\alpha\in (0,1]$ and let $X_1,\ldots,X_n$ be independent, identically distributed random variables with mean $\mu$ and $(1+\alpha)$-th central moment $M=\EXP\left[|X_i-\mu|^{1+\alpha}\right]$.
Let $m,k$ be positive integers and assume that $n=mk$. Then the median-of-means
estimator with $k =\left\lceil 8 \log(2/\delta) \right\rceil$ blocks satisfies
\[
   \PROB\left\{  \left|\wh\mu_n-\mu \right| > 8 \left(\frac{12M^{1/\alpha}\log(1/\delta)}{n}\right)^{\alpha/(1+\alpha)} \right\} \le \delta~.
\]
Moreover, for any mean estimator $\wh\mu_n$, there exists a distribution
with mean $\mu$ and  $(1+\alpha)$-th central moment $M$ such that
\[
   \PROB\left\{  \left|\wh\mu_n-\mu \right| >  \left(\frac{M^{1/\alpha}\log(2/\delta)}{n}\right)^{\alpha/(1+\alpha)} \right\} \ge \delta~.
\]
\end{theorem}

The proof of the first part follows by showing that if $c(\alpha)$ is
an appropriate constant that depends only on $\alpha$ and
$$
\eta \geq c(\alpha) \left(\EXP |X_i-\mu|^{1+\alpha}\right)^{1/(1+\alpha)} \left(\frac{1}{m}\right)^{\alpha/(1+\alpha)},
$$
then
$$
\PROB\left( \left|\frac{1}{m}\sum_{i=1}^m X_i - \mu \right| \geq \eta\right) \leq 0.2~.
$$
The proof of the second statement goes along the lines of Theorem \ref{thm:infvar}.

We finish this section by showing that if the distribution of $X$ has
a finite moment of order $2+\alpha$ for some $\alpha>0$, then the
median-of-means estimator has a sub-Gaussian performance under a much
wider range of 
choices for the parameter $k$ that counts the number of blocks. The following bound is due to Minsker and
Strawn \cite{MiSt17}. For simplicity of the exposition, we only consider the case $\alpha=1$.

\begin{theorem}
\label{thm:MiSt}
Let $X_1,\ldots,X_n$ be independent, identically distributed random variables
with mean $\mu$, variance $\sigma^2$, and third central moment $\rho=\EXP|X-\mu|^3$.
Let $m,k$ be positive integers and assume that $n=mk$.
Assume that 
\begin{equation}
\label{eq:cond}
\sqrt{\frac{\log(2/\delta)}{2k}} + \frac{\rho}{2\sigma^3\sqrt{m}} \le
1/4~.
\end{equation}
 Then the median-of-means
estimator $\wh\mu_n$ with $k$ blocks satisfies that, with probability
at least $1-\delta$,
\[
\left|\wh\mu_n-\mu\right| \le  \frac{1}{c} \left( \sigma \sqrt{\frac{\log(2/\delta)}{2n}} +
  \frac{\rho k}{2\sigma^2 n} \right)~,
\]
where $c=\phi(\Phi^{-1}(3/4))$ is a constant. Here $\phi$ and
$\Phi$ denote the standard normal density and distribution functions.
\end{theorem}

Observe that the first term on the right-hand side of the bound is of
the sub-Gaussian form. The second term is smaller than the first
whenever the number $k$ of blocks satisfies 
\[
    k\le \frac{2\sigma^3}{\rho}\sqrt{n\log(2/\delta)}~.
\]
In particular, $k\le \frac{2\sigma^3}{\rho}\sqrt{n}$ suffices to get a
sub-Gaussian performance. This is nice since with such a choice
the estimator does not depend on the value of the confidence parameter
$\delta$ and the estimator is sub-Gaussian simultaneously for the
entire range of values of $\delta$ permitted by the condition
(\ref{eq:cond}).
 Also, note that the number of blocks may be chosen to be
much larger than the choice suggested by Theorem \ref{thm:mom}. 
In particular, $k$ can be as large as a constant multiple of
$\sqrt{n}$. In that case the median-of-means estimator is
sub-Gaussian simultaneously for all $\delta \ge e^{-c_0\sqrt{n}}$ for an
appropriate constant $c_0$.
The price to pay is the extra assumption of the existence of the third
moment. Minsker and Strawn \cite{MiSt17} also prove that, when
$k=o(\sqrt{n})$,
then, under the assumptions of Theorem \ref{thm:MiSt},
$\sqrt{n}\left(\wh\mu_n-\mu\right)$ is asymptotically normal with mean
zero and variance $\sigma^2\pi/2$.

\begin{proof}
Note that $\wh\mu_n \in [\mu-a,\mu+a]$ if $a>0$ is such that
\[
    \frac{1}{k} \sum_{j=1}^k \IND_{Z_j-\mu \le a} \ge \frac{1}{2}
    \quad \text{and} \quad 
    \frac{1}{k} \sum_{j=1}^k \IND_{Z_j-\mu \ge -a} \ge \frac{1}{2}~.
\]
We show that, with probability at least $1-\delta$, one may take
\[
a=\frac{1}{c} \left( \sigma \sqrt{\frac{\log(2/\delta)}{2n}} +
  \frac{\rho k}{2\sigma^2 n} \right)~.
\]
To this end, note that
\begin{eqnarray*}
    \frac{1}{k} \sum_{j=1}^k \IND_{Z_j-\mu \le a}
& =&
  \frac{1}{k} \sum_{j=1}^k \left( \IND_{Z_j-\mu \le a}   -
     \PROB\left\{ Z_j-\mu \le a \right\} \right) \\
& & +  \left( \PROB\left\{ Z_1-\mu \le a \right\} 
          - \PROB\left\{ G \frac{\sigma}{\sqrt{m}} \le a \right\} \right) \\
& & + \PROB\left\{ G \frac{\sigma}{\sqrt{m}} \le a \right\} 
\\
& & \text{(where $G$ is a standard normal random variable).}
\end{eqnarray*}
First note that, by Hoeffding's inequality, with probability at least $1-\delta/2$,
\[
    \frac{1}{k} \sum_{j=1}^k \left( \IND_{Z_j-\mu \le a}   -
     \PROB\left\{ Z_j-\mu \le a \right\} \right)  \ge - \sqrt{\frac{\log(2/\delta)}{2k}}~.
\]
For the second term on the right-hand side, we may use the
Berry-Esseen theorem (see  Shevtsova \cite{She14}) that implies that
\[
 \PROB\left\{ Z_1-\mu \le a \right\} 
          - \PROB\left\{ G \frac{\sigma}{\sqrt{m}} \le a \right\}  \ge
          - \frac{\rho}{2\sigma^3\sqrt{m}}~.
\]
Hence, we have that, with probability at least  $1-\delta/2$,
\[
    \frac{1}{k} \sum_{j=1}^k \IND_{Z_j-\mu \le a} \ge \PROB\left\{ G
      \frac{\sigma}{\sqrt{m}} \le a \right\}-
    \sqrt{\frac{\log(2/\delta)}{2k}}
- \frac{\rho}{2\sigma^3\sqrt{m}}~.
\]
Thus,  $(1/k)\sum_{j=1}^k \IND_{Z_j-\mu \le a} \ge \frac{1}{2}$ with probability at least  $1-\delta/2$, whenever
$a$ is such that 
\[
   \PROB\left\{ G
      \le a \frac{\sqrt{m}}{\sigma}  \right\} \ge \frac{1}{2} +
    \sqrt{\frac{\log(2/\delta)}{2k}} + \frac{\rho}{2\sigma^3\sqrt{m}}~.
\]
If $\sqrt{\frac{\log(2/\delta)}{2k}} + \frac{\rho}{2\sigma^3\sqrt{m}}
\le 1/4$ then it suffices to consider values of $a$ with
 $a\sqrt{m}/\sigma \le \Phi^{-1}(3/4)$. Then 
\[
   \PROB\left\{ G
      \le a \frac{\sqrt{m}}{\sigma}  \right\} \ge \frac{1}{2} +
c \frac{a\sqrt{m}}{\sigma}
\]
with $c=\phi(\Phi^{-1}(3/4))$.
Hence, we may take 
\[
    a= \frac{\sigma}{c\sqrt{m}} \left(
      \sqrt{\frac{\log(2/\delta)}{2k}} +
      \frac{\rho}{2\sigma^3\sqrt{m}} \right) 
= \frac{1}{c} \left( \sigma \sqrt{\frac{\log(2/\delta)}{2n}} +
  \frac{\rho k}{2\sigma^2 n} \right)~.
\]
The same argument shows that, with probability at least $1-\delta/2$,
\[
   \frac{1}{k} \sum_{j=1}^k \IND_{Z_j-\mu \ge -a} \ge \frac{1}{2}
\]
for the choice of $a$ above.
\end{proof}

\subsection{Catoni's estimator}

Next we present a completely different approach for constructing a mean estimator, introduced and analyzed by Catoni \cite{Cat10}.
To introduce Catoni's idea, note first that the empirical mean $\ol\mu_n$ is just the solution
$y\in \R$ of the equation
\[
     \sum_{i=1}^n \left( X_i-y\right) = 0~.
\]
Catoni proposed to replace the left-hand side of the equation above by another strictly
decreasing function of $y$ of the form
\[
    R_{n,\alpha}(y) = \sum_{i=1}^n \psi\left(\alpha( X_i-y)\right)~,
\]
where $\psi:\R\to \R$ is an antisymmetric increasing function and $\alpha\in\R$ is a parameter.
The idea is that if $\psi(x)$ increases much slower than $x$, then the effect of ``outliers''
present due to heavy tails is diminished. Catoni offers a whole range of ``influence'' functions $\psi$.
For the ease of exposition, we single out one specific choice, namely
\[
    \psi(x) = \left\{ \begin{array}{ll}
                              \log(1+x+x^2/2)& \text{if} \ x\ge 0 \\
                              -\log(1-x+x^2/2) & \text{if} \ x < 0~.
                             \end{array}  \right.
\]
We define Catoni's mean estimator $\wh\mu_{\alpha,n}$ as the unique value $y$ such that $R_{n,\alpha}(y)=0$
with this choice of $\psi$.  Since $\psi(x)\le \log(1+x+x^2/2)$ for all $x\in \R$, we have, for all $y\in \R$,
\begin{eqnarray*}
\EXP \left[ e^{R_{n,\alpha}(y)} \right] & \le & \left( \EXP\left[ 1+ \alpha (X-y) + \frac{\alpha^2 (X-y)^2}{2} \right]\right)^n\\
& = & \left( 1 + \alpha(\mu-y) + \frac{\alpha^2 \left(\sigma^2+ (\mu-y)^2\right)}{2} \right)^n \\
& \le & \exp\left( n\alpha(\mu-y) + \frac{n\alpha^2 \left(\sigma^2+ (\mu-y)^2\right)}{2}\right)~,
\end{eqnarray*}
whenever the $X_i$ have a finite variance $\sigma^2$. Thus, by Markov's inequality, we have that, for any fixed $y\in \R$
and $\delta \in (0,1)$,
\[
   \PROB\left\{ R_{n,\alpha}(y) \ge n\alpha(\mu-y) + \frac{n\alpha^2 \left(\sigma^2+ (\mu-y)^2\right)}{2} + \log(1/\delta)\right\}
     \le \delta~.
\]
Suppose that the parameter $\alpha$ is such that $\alpha^2\sigma^2 + 2\log(1/\delta)/n \le 1$.
Then the quadratic polynomial of $y$ 
\[
n\alpha(\mu-y) + \frac{n\alpha^2 \left(\sigma^2+ (\mu-y)^2\right)}{2}
+ \log(1/\delta)
\]
has at least one root.
In particular, taking the smaller root
\[
    y_+ = \mu + \frac{\frac{\alpha\sigma^2}{2}  + \frac{\log(1/\delta)}{n\alpha}}{\frac{1}{2}+ \frac{1}{2}\sqrt{1-\alpha^2\sigma^2 - \frac{2\log(1/\delta)}{n}}}~,
\]
we have that $R_{n,\alpha} (y_+) < 0$ with probability at least $1-\delta$.
Since $R_{n,\alpha}(y)$ is strictly decreasing, this implies that $\wh\mu_{\alpha,n} < y_+$ with probability at least $1-\delta$.
A symmetric argument shows that $\wh\mu_{\alpha,n} > y_-$ with probability at least $1-\delta$,
where
\[
    y_- = \mu - \frac{\frac{\alpha\sigma^2}{2}  + \frac{\log(1/\delta)}{n\alpha}}{\frac{1}{2}+ \frac{1}{2}\sqrt{1-\alpha^2\sigma^2 - \frac{2\log(1/\delta)}{n}}}~.
\]
Now by straightforward bounding, and choosing the parameter $\alpha$ to optimize the bounds, we obtain
the following performance estimate.

\begin{theorem}
\label{thm:catoni}
Let $X_1,\ldots,X_n$ be independent, identically distributed random variables
with mean $\mu$ and variance $\sigma^2$. Let $\delta\in (0,1)$ be such that $n> 2\log(1/\delta)$.
Catoni's mean estimator $\wh\mu_{n,\alpha}$ with parameter
\[
     \alpha= \sqrt{\frac{2\log(1/\delta)}{n\sigma^2\left(1+\frac{2\log(1/\delta)}{n-2\log(1/\delta)}\right)}}
\]
satisfies that, with probability at least $1-2\delta$,
\begin{equation}
\label{eq:catonibound}
     \left| \wh\mu_{n,\alpha} - \mu \right| < \sqrt{\frac{2\sigma^2 \log(1/\delta)}{n-2\log(1/\delta)}}~.
\end{equation}
\end{theorem}

The theorem highlights that, with an appropriately chosen parameter $\alpha$, Catoni's estimator has
a sub-Gaussian performance. Quite remarkably, the constant $\sqrt{2}$ is the best possible. A disadvantage
of Catoni's estimator with respect to median-of-means is that the
estimator---at least in the form given in the theorem---depends on the variance $\sigma^2$. In general, it is unrealistic to assume knowledge of $\sigma^2$.
If one substitutes $\sigma^2$ in the formula of $\alpha$ by an upper bound $v$, then the bound (\ref{eq:catonibound})
still holds with $v$ replacing $\sigma^2$. In case no good upper bound for $\sigma^2$ is available,
Catoni \cite{Cat10} shows how to use Lepski's method to select $\alpha$ from the data that has near-optimal performance.
Huber \cite{Hub19} combines the median-of-means estimator with Catoni's estimator into a two-step
procedure that to obtain an estimator with the optimal leading constant in the sub-Gaussian bound when
$|\sigma/\mu|$ is bounded by a known constant.

Another problem---shared with the median-of-means estimator---is that Catoni's estimator also depends
on the required confidence level $\delta$. Such a dependence is
necessary as it is shown in Section \ref{sec:multidelta} below.
A quick fix is to use the
estimator with a $\delta$-independent parameter, 
though then the resulting estimate, naturally, cannot be sub-Gaussian.
One reasonable choice is $\alpha=\sqrt{2/(n\sigma^2)}$. In this case, it is easy to see that, whenever $n>2(1+\log(1/\delta))$,
Catoni's estimator satisfies, with probability at least $1-2\delta$,
\[
    \left| \wh\mu_{n,\alpha} - \mu \right| < \sqrt{\frac{\sigma^2}{2n}} \cdot \frac{1+ \log(1/\delta)}{1-\frac{1+ \log(1/\delta)}{n}}~.
\]
This is not a sub-Gaussian bound because of an extra factor of $\sqrt{\log(1/\delta)}$ but the ``sub-exponential'' tail
probabilities are still non-trivial and useful.

\subsection{Trimmed mean}

Perhaps the most natural attempt to improve the performance of the empirical mean is 
removing possible outliers using a truncation of $X$. Indeed, the
so-called \emph{trimmed-mean} (or \emph{truncated-}mean) estimator
is defined by removing a fraction of the sample,
consisting of the $\epsilon n$ largest and smallest points for some
parameter $\epsilon \in (0,1)$, and then averaging over the rest. 
This idea is one of the most classical tools in robust statistics and
we refer to  Tukey and McLaughlin \cite{TuMc63}, Huber and Ronchetti
\cite{HuRo09}, Bickel \cite{Bic65}, Stigler \cite{Sti73} for early
work on the theoretical properties of the trimmed-mean estimator.
However, it was only recently that the non-asymptotic sub-Gaussian 
property of the trimmed mean was established. Indeed, Oliveira and
Orenstein \cite{OlOr19} proved that if $\epsilon$ is chosen
proportionally to $\log(1/\delta)/n$, then 
the trimmed-mean estimator has a sub-Gaussian performance for all distributions with a
finite variance.

To show how this works in the simplest way, here we analyze a simple
variant of the trimmed-mean estimator.

The estimator splits the data in two
equal parts. One half is used to determine the correct
truncation level. The points from the other half are
averaged, except for the data points
that fall outside of the truncation region, which are ignored.
For convenience of the notation, we assume that the 
data consists of $2n$ independent copies of the random variable $X$,
denoted by  $X_1,\ldots,X_n,Y_1,\ldots,Y_n$. 

For $\alpha \leq \beta$, define the truncation function
\[
\phi_{\alpha,\beta}(x) =
\begin{cases}
\beta & \mbox{if} \  x > \beta,
\\
x & \mbox{if} \ x \in [\alpha,\beta]~,
\\
\alpha & \mbox{if} \ x < \alpha~,
\end{cases}
\]
and for $x_1,\ldots, x_m \in \R$ let $x_1^* \leq x_2^* \leq \cdots \leq  x_m^*$ be its non-decreasing rearrangement.

With this notation in place, the definition of the estimator is as follows:

\begin{description}
\item{$(1)$} Given the confidence level $\delta\ge 8e^{-3n/16}$, set
$$
\eps=\frac{16\log(8/\delta)}{3n}~.
$$
\item{$(2)$} Let $\alpha=Y_{\eps n}^*$ and $\beta=Y_{(1-\eps) n}^*$
  (assuming, for simplicity, that $\eps n$ is an integer) and set
$$
\wh{\mu}_{2n} =\frac{1}{n}\sum_{i=1}^n \phi_{\alpha,\beta}(X_i)~.
$$
\end{description}

\begin{theorem} 
\label{thm:trimmed-mean}
Let $X_1,\ldots,X_n,Y_1,\ldots,Y_n$ be independent, identically distributed random variables
with mean $\mu$ and variance $\sigma^2$. Let $\delta\in (0,1)$ be such that $n> (16/3)\log(8/\delta)$.
Then, with probability at least $1-\delta$,
$$
|\wh{\mu}_{2n}-\mu | \leq 9 \sigma \sqrt{\frac{\log(8/\delta)}{n}}~.
$$
\end{theorem}

\begin{proof}
We start by showing that the truncation level is close to the
appropriate quantiles of the distribution. To this end, for $p\in
(0,1)$, introduce the quantiles
\[
Q_p = \sup \left\{M \in \R : \PROB\left\{X \geq M\right\} \geq 1-p\right\}~.
\]
For ease of exposition, assume that $X$ has a nonatomic distribution. 
(This assumption is not necessary, but simplifies notation.)
In that case $\PROB\{X > Q_p\} = \PROB\{X \ge Q_p\} = 1-p$.

By a straightforward application of Bernstein's inequality, with
probability at least $1-2\exp(-(3/16)\eps n)$,
we have both
$$
\left|\{i\in [n]: Y_i \geq Q_{1-2\eps}\}\right| \geq \eps n
$$
and
$$
\left|\{i\in [n]: Y_i \leq Q_{1-\eps/2}\}\right| \geq (1-\eps) n~.
$$
This implies that, with probability at least $1-2\exp(-(3/16)\eps n)$,
\begin{equation} 
\label{eq:beta-single}
Q_{1-2\eps} \leq Y_{(1-\eps) n}^* \leq  Q_{1-\eps/2}~.
\end{equation}
By the same argument, with probability at least $1-2\exp(-(3/16)\eps n)$,
\begin{equation} \label{eq:alpha-single}
Q_{\eps/2} \leq Y_{\eps n}^* \leq Q_{2\eps}~,
\end{equation}
From here, we simply need to show that
$|\E \phi_{\alpha,\beta}(X)- \mu|$ is small and that 
 $(1/n)\sum_{i=1}^n \phi_{\alpha,\beta}(X_i)$ concentrates around its mean.

For the first step, consider the event $E$ that both \eqref{eq:beta-single} and \eqref{eq:alpha-single}
hold. This event has probability at least $1-4\exp(-(3/16)\eps n) =1-\delta/2$.
On the event $E$,
\begin{eqnarray*}
\lefteqn{
 \left| \E
  \left[\phi_{\alpha,\beta}(X) |Y_1,\ldots,Y_n\right]- \mu
  \right|} \\
& \le & \left|\E \left[(X-\alpha)\IND_{X\leq
  \alpha}|Y_1,\ldots,Y_n\right] \right|+
 \left|\E \left[(X-\beta)\IND_{X\geq
  \beta}|Y_1,\ldots,Y_n\right] \right|
\\
& \leq & |\E (X-Q_{2\eps})\IND_{X\leq Q_{2\eps}}|+ 
|\E  (X-Q_{1-2\eps})\IND_{X\geq Q_{1-2\eps}}|~.
\end{eqnarray*}
To bound these two terms, forst notice that, by Chebyshev's inequality,
$$
2\eps = \PROB\left\{ X \geq Q_{1-2\eps}\right\} \leq \frac{\sigma_X^2}{(Q_{1-2\eps}-\mu)^2}~,
$$
and in particular,
$$
Q_{1-2\eps} \leq \mu + \frac{\sigma}{\sqrt{2\eps}}~.
$$
Hence, by the Cauchy-Schwarz inequality,
\begin{eqnarray*}
|\E  (X-Q_{1-2\eps})\IND_{X\geq Q_{1-2\eps}}| 
& = &
|\E (X-\mu)-(Q_{1-2\eps}-\mu)) \IND_{X \geq Q_{1-2\eps}}| \\
& \leq  &
\E |(X-\mu)| \IND_{X \geq Q_{1-2\eps}} + (Q_{1-2\eps}-\mu)\PROB\{X \geq
          Q_{1-2\eps}\} \\
& \leq &  
\sigma \sqrt{\PROB\left\{X \geq Q_{1-2\eps}\right\}} +  2\eps (Q_{1-2\eps}-\mu) \\
& \leq & \sigma\sqrt{8\eps}~.
\end{eqnarray*}
A symmetric argument shows $|\E (X-Q_{2\eps})\IND_{X\leq Q_{2\eps}}|
\le \sigma\sqrt{8\eps}$, and therefore, on the event $E$, we have
\[
 \left| \E
  \left[\phi_{\alpha,\beta}(X) |Y_1,\ldots,Y_n\right]- \mu
  \right| \le \sigma\sqrt{32\eps} \le  6\sigma \sqrt{\frac{\log(8/\delta)}{n}}
\]
by our choice of $\epsilon$. 
Next, let
$$
Z=\frac{1}{n} \sum_{i=1}^n \phi_{\alpha,\beta}(X_i)-\E\left[ \phi_{\alpha,\beta}(X) |Y_1,\ldots,Y_n\right]
$$
and observe that
$$
Z= \frac{1}{n} \sum_{i=1}^n \phi_{\alpha-\mu,\beta-\mu}(X_i-\mu)-\E\left[ \phi_{\alpha-\mu,\beta-\mu}(X-\mu) |Y_1,\ldots,Y_n\right]~.
$$
Hence, on the event $E$ (that only depends on $Y_1,\ldots,Y_n$),  
$Z$ is an average of  centered random variables that is bounded point-wise 
by $M=\max \{|Q_{\eps/2}-\mu|,|Q_{1-\eps/2}-\mu|\} \le
\sigma\sqrt{2/\eps}$ and whose
variance is at most $\sigma^2$. Therefore, by Bernstein's inequality, with probability at least $1-\delta/2$,
$$
Z \leq \sigma \sqrt{\frac{2\log(2/\delta)}{n}}
+\frac{\log(2/\delta)\sigma\sqrt{2/\eps}}{n}
\le 3\sigma \sqrt{\frac{\log(2/\delta)}{n}}
~.
$$
Putting the pieces together, we obtain
the announced bound.
\end{proof}

Besides its conceptual simplicity, an important advantage of the trimmed mean compared to other
estimators with sub-Gaussian performance is that it is robust to
adversarial contamination of the data. This statement is formalized
and proved in \cite{LuMe19a} where a multivariate extension is also
introduced and analyzed.

\subsection{Multiple-$\delta$ estimators}
\label{sec:multidelta}

We have constructed various estimators--such as median-of-means and
Catoni's estimator--that are sub-Gaussian under the
only assumption that the underlying distribution has a finite second
moment.  However, both estimators depend on the knowledge of the 
desired confidence parameter $\delta$. We show next that is not a
coincidence because 
without
further information on the distribution, it is impossible to construct a single estimator that is sub-Gaussian
for a nontrivial range of values of the confidence parameter $\delta$.
Next we reproduce a simplified version of an argument of Devroye, Lerasle, Lugosi,
and Oliveira \cite{DeLeLuOl16} who proved results of this kind. The theorem below shows that
it is impossible to construct an estimator that is $L$-sub-Gaussian for some specified values of
$\delta_1$ and $\delta_2$, at the same time. The particular values of
$\delta_1$ and $\delta_2$ are of
no special importance. We present this result to show the basic ideas in a simple form. For more
general versions we refer to \cite{DeLeLuOl16}.

\begin{theorem}
\label{thm:multidelta}
For every $L\ge 50$ and for every sample size $n$, no
estimator can be simultaneously $L$-sub-Gaussian for both
$\delta_1=1/(2e\sqrt{L^3+1})$ and $\delta_2=2e^{-L^4/4}$ for
all distributions with finite second moment. 
\end{theorem}

\begin{proof}
We show that not only it is impossible to construct a single
$L$-sub-Gaussian estimator for both
$\delta_1=1/(2e\sqrt{L^3+1})$ and $\delta_2=e^{-L^4/4}$ for
all distributions with finite second moment but it is also the case
for the restricted class of Poisson distributions. 

Assume, on the contrary, that there exists an estimator $\wh{\mu}_n$
that is $L$-sub-Gaussian for both $\delta_1$ and $\delta_2$ for all
Poisson distributions. Let $X_1,\ldots,X_n$ be independent
Poisson random variables with parameter $1/n$ and let $Y_1,\ldots,Y_n$
be independent Poisson random variables with parameter $c/n$, where we
set $c=L^3+1$. We assume, for the sake of simplicity, that $c$ is an integer.
 By the sub-Gaussian property of $\wh{\mu}_n$,
\begin{equation}
\label{eq:Poi1}
     \PROB\left\{ \wh{\mu}_n(Y_1,\ldots,Y_n) < \frac{c}{n} -
       \frac{L}{n}\sqrt{c\log\frac{1}{\delta_1}} \right\} \le \delta_1~.
\end{equation}
Now note that the left-hand side of the inequality may be lower bounded as folows:
\begin{eqnarray*}
\lefteqn{
\PROB\left\{ \wh{\mu}_n(Y_1,\ldots,Y_n) < \frac{c}{n} -
       \frac{L}{n}\sqrt{c\log\frac{1}{\delta_1}} \right\}  }    \\
& \ge &
\PROB\left\{ \wh{\mu}_n(Y_1,\ldots,Y_n) < \frac{c}{n} -
       \frac{L}{n}\sqrt{c\log\frac{1}{\delta_1}}, \sum_{i=1}^n Y_i = c \right\} \\
& \ge &
\frac{1}{e\sqrt{c}}\PROB\left\{ \wh{\mu}_n(Y_1,\ldots,Y_n) < \frac{c}{n} -
       \frac{L}{n}\sqrt{c\log\frac{1}{\delta_1}} \Big{|} \sum_{i=1}^n Y_i = c \right\} \\
& & \text{(from the fact that $ \sum_{i=1}^n Y_i$ is Poisson with parameter $c$ and Stirling's formula)}
\end{eqnarray*}
Next we use the fact that the conditional joint distribution of $n$ independent 
Poisson($\lambda$) random variables, conditioned on the event that 
 their sum equals $c$, only depends on $c$ but not $\lambda$. In particular,
\begin{eqnarray*}
\lefteqn{
\PROB\left\{ \wh{\mu}_n(Y_1,\ldots,Y_n) < \frac{c}{n} -
       \frac{L}{n}\sqrt{c\log\frac{1}{\delta_1}} \Big{|} \sum_{i=1}^n Y_i = c \right\}    }\\
& = & \PROB\left\{ \wh{\mu}_n(X_1,\ldots,X_n) < \frac{c}{n} -
       \frac{L}{n}\sqrt{c\log\frac{1}{\delta_1}} \Big{|} \sum_{i=1}^n X_i = c \right\}~.
\end{eqnarray*}
Thus, together with (\ref{eq:Poi1}), and the choice $\delta_1=
1/(2e\sqrt{c})$. we have that
\begin{eqnarray*}
\frac{1}{2} &=&    1- e\sqrt{c} \delta_1 \\
 & \le &
\frac{ \PROB\left\{ \wh{\mu}_n(X_1,\ldots,X_n) \ge \frac{c}{n} -
       \frac{L}{n}\sqrt{c\log\frac{1}{\delta_1}} , \sum_{i=1}^n X_i = c \right\}  }
{\PROB\left\{ \sum_{i=1}^n X_i = c \right\} } 
\\
& \le & e c! \PROB\left\{ \wh{\mu}_n(X_1,\ldots,X_n) \ge \frac{c}{n} -
       \frac{L}{n}\sqrt{c\log\frac{1}{\delta_1}} \right\}  \\
& \le & e c! \PROB\left\{ \wh{\mu}_n(X_1,\ldots,X_n) \ge \frac{1}{n} + \frac{c-1}{n} -
       \frac{L}{n}\sqrt{c\log\frac{1}{\delta_1}}\right\}  \\
& \le & e c! \PROB\left\{ \wh{\mu}_n(X_1,\ldots,X_n) \ge \frac{1}{n} + \frac{c-1}{2n} \right\}~,
\end{eqnarray*}
where we used the fact that 
\[
       \frac{L}{n}\sqrt{c\log\frac{1}{\delta_1}} \le \frac{c-1}{2n}~,
\]
that follows from our choice of $\delta_1$ whenever $L\ge 10$.
Now since $\wh{\mu}_n$
is $L$-sub-Gaussian for $\delta_2=2e^{-L^4/4}$, we have that
\[
   \PROB\left\{ \wh{\mu}_n(X_1,\ldots,X_n) \ge \frac{1}{n} +
     \frac{c-1}{2n} \right\}   =
   \PROB\left\{ \wh{\mu}_n(X_1,\ldots,X_n) \ge \frac{1}{n} +
     \frac{L}{n}\sqrt{\log(2/\delta_2)} \right\}   \le \delta_2~.
\]
Summarizing, we have $1/2\le e c! \delta_2 = 2e c!
e^{-L^4/4}$. However, the expression on the right-hand side is less
than $1/2$ for $L\ge 50$, leading to a contradiction. 
\end{proof}

We refer to \cite{DeLeLuOl16} for a more complete version of Theorem
\ref{thm:multidelta} and for an extensive discussion on constructing estimators
that do not require knowledge of the desired confidence parameter (i.e., estimators that are sub-Gaussian for a wide range of
values of $\delta$). In \cite{DeLeLuOl16} it is shown how Lepski's method may be used to construct such estimators if some additional information,
other than finiteness of the variance, is available on the underlying distribution. In particular, if nontrivial upper and lower
bounds on the variance are available, then such
``$\delta$-independent'' estimators exist for a wide range of values
of $\delta$. Existence of higher moments
or certain weak symmetry assumptions may also be used. 

\section{Estimating the mean of a random vector}
\label{sec:vectorest}

In what follows, we discuss extensions of the mean estimation problem
to the multivariate setting. To set up the problem, let $X$ be a random vector taking values in $\R^d$.
Assume that the mean vector $\mu = \EXP X$
and covariance matrix $\Sigma= \EXP (X-\mu) (X-\mu)^T$ exist.
Given $n$ independent, identically distributed samples
$X_1,\ldots,X_n$ drawn from the distribution of $X$, one
wishes to estimate the mean vector.

Just like in the univariate case,
a natural choice is the sample mean $\ol{\mu}_n=(1/n)\sum_{i=1}^n X_i$ and it has a near-optimal behavior whenever the
distribution is sufficiently light tailed. However, as is the case in the univariate case,
whenever
heavy tails are a concern, the sample mean is to be avoided
as it may have a sub-optimal performance.

\subsection{Sub-Gaussian performance}

For the univariate problem, we constructed mean estimators with a
sub-Gaussian performance. In order to properly set up our goal for the
$d$-dimensional case, first we need to understand what ``sub-Gaussian
performance'' means. Just like in the univariate case, one would like to construct
estimators that are ``close'' to the true mean $\mu$, with ``high
probability''.  The first question is how one measures distance in
$\Rd$. Arguably, the most natural distance measure is the Euclidean
norm. 
In this section we focus on this choice and we denote by $\| \cdot \|$ 
the Euclidean norm.  We explore mean estimation of a
random vector with respect to an arbitrary norm in Section
\ref{sec:vector-MOM-general}.  

If $X$ has a multivariate normal distribution with mean vector $\mu$
and
covariance matrix $\Sigma$, then the sample mean $\ol{\mu}_n$ is also multivariate
normal with mean $\mu$ and covariance matrix $(1/n)\Sigma$. Thus, for all $t>0$,
\[
   \PROB\left\{ \|\ol{\mu}_n-\mu\|\ge \EXP \|\ol{\mu}_n-\mu\| + t \right\}
= \PROB\left\{ \left\|\ol{X} \right\|- \EXP \left\|\ol{X} \right\| \ge   t\sqrt{n} \right\}~,
\]
where $\ol{X}$ is a Gaussian vector in $\R^d$ with zero mean and covariance matrix $\Sigma$.
A key property of Gaussian vectors is that $\ol{X}$ has the same distribution as $\Sigma^{1/2}Y$
where $Y$ is a standard normal vector (i.e., with zero-mean and identity covariance matrix) and
$\Sigma^{1/2}$ is the positive semidefinite square root of $\Sigma$.
Also, observe that for all $y,y'\in \Rd$,
\[
    \left| \left\|\Sigma^{1/2}y\right\|- \left\|\Sigma^{1/2}y'\right\| \right| \le \left\|\Sigma^{1/2}(y-y')\right\| \le \left\|\Sigma^{1/2}\right\|_{2 \to 2} \cdot \|y-y'\|~,
\]
where $\left\|\Sigma^{1/2}\right\|_{2 \to 2}$ is the spectral norm of $\Sigma^{1/2}$. Thus,
$\Sigma^{1/2}y$ is a Lipschitz function of $y\in \Rd$ with Lipschitz constant $\|\Sigma^{1/2}\|_{2 \to 2}= \sqrt{\lambdamax}$,
with $\lambdamax=\lambdamax(\Sigma)$ denoting the largest eigenvalue of the covariance matrix $\Sigma$.
Now it follows from the Gaussian concentration inequality
of Tsirelson, Ibragimov, and Sudakov~\cite{TsIbSu76}
(see also Ledoux \cite{Led01} and Boucheron, Lugosi, and Massart
\cite{BoLuMa13} for more information)
that
\[
\PROB\left\{ \left\|\ol{X} \right\|- \EXP \left\|\ol{X} \right\| \ge   t\sqrt{n} \right\}  \le e^{-nt^2/(2\lambdamax)}~.
\]
Noting that
\[
   \EXP \left\|\ol{X} \right\| \le \sqrt{ \EXP \left\|\ol{X} \right\|^2} = \sqrt{\Tr(\Sigma)}~,
\]
the trace of the covariance matrix $\Sigma$, we have that, for $\delta \in (0,1)$, with probability at least $1-\delta$,
\begin{equation}
\label{eq:subgaussmulti}
   \|\ol{\mu}_n-\mu\| \le \sqrt{\frac{\Tr(\Sigma)}{n}}
    + \sqrt{\frac{2\lambdamax \log(1/\delta)}{n}}~.
\end{equation}
Thus, in the multivariate case, we will say that a mean estimator is
\emph{sub-Gaussian} if, with probability at least $1-\delta$, it
satisfies an inequality of the form (\ref{eq:subgaussmulti}) (with
possibly different constant factors).
Note that for any distribution with mean $\mu$ and covariance matrix $\Sigma$, the
mean-squared error of the empirical mean equals
\[
     \EXP \|\ol{\mu}_n-\mu\|^2 = \frac{\Tr(\Sigma)}{n}~.
\]
In particular, $\EXP \|\ol{\mu}_n-\mu\| \le \sqrt{\frac{\Tr(\Sigma)}{n}}$. An important feature of the sub-Gaussian
property (\ref{eq:subgaussmulti}) is that the random fluctuations are controlled by the spectral norm $\lambdamax$
of the covariance matrix, which is possibly much smaller than $\Tr(\Sigma)$, the sum of all eigenvalues of $\Sigma$.

\subsection{Multivariate median-of-means}

For non-Gaussian and possibly
heavy-tailed distributions, one cannot expect a sub-Gaussian
behavior of the sample mean similar to (\ref{eq:subgaussmulti}).

As an alternative, one may try to extend the median-of-means estimator to the multivariate
case. An obvious idea is to divide the data into disjoint blocks, calculate the
empirical mean within each block, and compute a multivariate median of the 
obtained empirical means. However, there is no standard notion of a
median for multivariate data, and it is not entirely clear what
definition of a multivariate median works best for median-of-means
mean estimators. Among the numerous possibilities, we mention
the \emph{coordinate-wise median}, the \emph{geometric (or spatial) median},
the \emph{Tukey (or halfspace) median}, the \emph{Oja median},
and the \emph{Liu median},
see Small \cite{Sma90} for a survey and relevant references.

Regardless of what notion of a multivariate median we decide to adopt, 
we start by partitioning $[n]=\{1,\dots,n\}$ into $k$ blocks $B_1,\ldots,B_k$, each of
size
$|B_i|\geq \lfloor n/k\rfloor\geq 2$. Here $k$ is a parameter of the estimator to be chosen later.
For simplicity, we assume that $km=n$ for some positive integer $m$.
Just like before,
we compute the sample mean of the random vectors within each block: for $j=1,\ldots,k$, let
$$
Z_j=\frac{1}{m}\sum_{i\in B_j}X_i~.
$$
Perhaps the most natural first try is to define $\wh{\mu}_n$ as the vector of coordinate-wise 
medians of the $Z_j$ (i.e., the $\ell$-th component of the vector $\wh{\mu}_n$ is the median 
of the $\ell$-th components of $Z_1,\ldots,Z_k$, for $\ell\in [d]$). Then Theorem \ref{thm:mom}
and the union bound imply that, for any $\delta\in (0,1)$, taking $k= \left\lceil 8 \log(1/\delta) \right\rceil$, with probability
at least $1-\delta$,
\[
   \left\|\wh\mu_n-\mu \right\| \le  \sqrt{\frac{32\Tr(\Sigma)\log(d/\delta)}{n}}~,
\]
where we used the fact that $\Tr(\Sigma)=\EXP\|X-\EXP X\|^2$ is the
sum of the variances of the $d$ components of $X$. Clearly, this bound
is far from the sub-Gaussian inequality  (\ref{eq:subgaussmulti}) for
several reasons. First, it is not ``dimension-free'' as $d$ appears
explicitly in the bound. Perhaps more importantly, $\log(1/\delta)$ is
multiplied by $\Tr(\Sigma)$ instead of $\lambdamax(\Sigma)$ and that may
be a major difference in high-dimensional problems, especially when one is
interested in small failure probabilities. An instructive example is
when all eigenvalues of $\Sigma$ are identical and equal to
$\lambdamax$.  If the dimension $d$ is large, \eqref{eq:subgaussmulti}
is of the order of $\sqrt{(\lambdamax/n)(d+\log (1/\delta))}$ while
the bound above only gives the order
$\sqrt{(\lambdamax/n)(d\log (d/\delta))}$.

One may quite easily improve on this by using a different (non-standard) notion of median 
in the definition of the estimate:
choose $\wh\mu_n$ to be the point in $\Rd$ with the property that the Euclidean ball centered at
$\wh\mu_n$ that contains more than $k/2$ of the points $Z_j$ has minimal radius.
Since $\EXP \|Z_j-\mu\|^2 = \Tr(\Sigma)/m$, by Chebyshev's inequality,
 $\|Z_j-\mu\| \le r \defeq 2\sqrt{\Tr(\Sigma)/m}$ with probability at least $3/4$. Thus, by choosing
$k= \left\lceil 8 \log(1/\delta) \right\rceil$, we have that, with probability at least $1-\delta$,
more than half of the points $Z_j$ satisfy
\[
   \|Z_j-\mu\| \le r~.
\]
Denote this event by $E$. (Thus, $\PROB\{E\}\ge 1-\delta$.)
 On the event $E$,
this radius is at most $r$. Hence, at least one of the $Z_j$ is within distance $r$ to both $\mu$ and $\wh\mu_n$.
Thus, by the triangle inequality, $\|\wh\mu_n-\mu\| \le 2r$. We have obtained the following proposition.

\begin{proposition}
\label{prop:mmm}
Let $X_1,\ldots,X_n$ be i.i.d.\ random vectors in $\Rd$ with mean $\mu$ and covariance matrix $\Sigma$.
Let $\delta\in (0,1)$ and
let $\wh\mu_n$ be the estimator defined above with $k= \left\lceil 8 \log(1/\delta)\right\rceil$.
Then, with probability at least $1-\delta$,
\[
  \left\|\wh\mu_n-\mu\right\| \le 4\sqrt{\frac{\Tr(\Sigma) (8 \log(1/\delta)+1)}{n}}~.
\]
\end{proposition}

The bound of Proposition \ref{prop:mmm} is quite remarkable as it is ``dimension-free'' and
no assumption other
than the existence of the covariance matrix is made. However, it still does not
achieve a sub-Gaussian performance bound that resembles
\eqref{eq:subgaussmulti}. Moreover, the notion of median used here (i.e., 
the center of the smallest ball that contains at least half of the points) is
problematic from a computational point of view, since computing such a median is a nontrivial problem.

An efficiently computable version of a multivariate median is the so-called \emph{geometric median}, defined as
\[
    \wh\mu_n = \argmin_{m\in \Rd} \sum_{j=1}^k \|Z_i-m\|~.
\]
This estimator was proposed by Minsker \cite{Min15} and independently by Hsu and Sabato \cite{HsSa16} (see also Lerasle and Oliveira \cite{LerasleOliveira_Robust}). In particular, Minsker \cite{Min15} proved that this version of
the multivariate median-of-means estimator achieves a similar performance bound as Proposition \ref{prop:mmm}.
Moreover, computing the geometric median---and therefore
the multivariate median-of-means estimator---involves solving a convex optimization problem. Thus, the geometric median may be approximated efficiently,
see Cohen, Lee, Miller, Pachocki, and Sidford \cite{CoLeMiPaSi16} for the most recent result and for the rich history of the problem.
We refer to Aloupis \cite{Alo06} for a survey of computational aspects of
various other notions of multivariate medians.

For a quite different mean estimator based on the median-of-means idea
with ``almost'' sub-Gaussian guarantees but with a serious computational
burden, see Joly, Lugosi, and Oliveira \cite{JoLuOl17}.

In order to achieve a truly sub-Gaussian performance, we need to
define a new estimator. In what follows we define two that achieve the desired performance: the
first, introduced in \cite{LuMe16a} is based on the idea of median-of-means
tournaments and the second, from \cite{LuMe18a}, is defined using the
intersection of random slabs. The former leads to an error estimate
with respect to the Euclidean norm (see Section \ref{sec:mmt}), and the latter, described in
Section \ref{sec:vector-MOM-general} holds with respect to an arbitrary norm.
However, before presenting these estimates, we recall a very different estimator 
introduced by Catoni and Giulini \cite{CaGi18}.

\subsection{Thresholding the norm: the Catoni-Giulini estimator}

In this section we briefly discuss a remarkably simple estimator,
suggested 
and analyzed by Catoni and Giulini \cite{CaGi18}.
The Catoni-Giulini estimator is 
\begin{equation}
\label{eq:cagi}
    \wh\mu_n = \frac{1}{n}\sum_{i=1}^n X_i
    \min\left(1,\frac{1}{\alpha \|X_i\|}  \right)~,
\end{equation}
where $\alpha>0$ is a (small) parameter. Thus, $\wh\mu_n$ is simply an
empirical average of the $X_i$, with the data points with large norm
shrunk towards zero. This estimate is trivial to compute, as opposed
to the more complex estimators that we discuss in Sections
\ref{sec:mmt} and \ref{sec:vector-MOM-general}. On the other hand, shrinking to
zero is somewhat arbitrary and unnatural. In fact, the estimator is
not invariant under translations of the data in the sense that 
$\wh\mu_n(X_1+a,\ldots,X_n+a)$ is not necessarily equal to
$\wh\mu_n(X_1,\ldots,X_n)+a$ when $a\neq 0$. 

Catoni and Giulini prove that if one chooses the
parameter 
as 
\[
\alpha=\sqrt{\frac{c\log(1/\delta)}{v n}}~,
\]
where $v\ge \lambdamax$ and $c>0$ is a numerical constant, 
then the estimator (\ref{eq:cagi}) satisfies, with probability at
least $1-\delta$,
\[
   \left\|\wh\mu_n - \mu\right\| \le C\sqrt{\frac{\left(\Tr(\Sigma)+ v+\|\mu\|^2\right)\log(1/\delta)}{n}}~,
\]
where $C$ is a constant depending on $c$ only. This bound is similar
to but weaker than that of Proposition \ref{prop:mmm}, principally due
to two facts. First, the estimator requires prior knowledge of (a good upper
bound of) $\lambdamax$ whereas the geometric median-of-means estimator 
assumes no such prior information. Second, $\|\mu\|^2$ appears in the
upper bound and a priori this can be arbitrarily large compared to
$\Tr(\Sigma)$. The presence of this term is due to the lack of
translation invariance of the estimator.
This second issue may be fixed by  defining a two-stage estimator: first one may
use a translation-invariant estimator like geometric-median-of-means defined in the
previous section to get a rough estimate of the mean. Then, using a
new batch of independent data, one may center the data at the
estimated mean and then use the Catoni-Giulini estimator for the
centered data. This new estimator is translation invariant, and the
term
$\|\mu\|^2$ may be replaced by the squared error of the estimator of
the first step, that is, by $\Tr(\Sigma)\log(1/\delta)/n$. But even with
this modification, the bound is not sub-Gaussian in the sense of
(\ref{eq:subgaussmulti}).

Remarkably, however, the performance of the Catoni-Giulini
estimator comes close to being sub-Gaussian in the desired sense under 
just a small extra assumption. In particular, if $\EXP \|X\|^\beta
<\infty$ for some $\beta>2$, then, with the same choice of $\alpha$ as
above, one has
\[
  \left\|\wh\mu_n - \mu\right\| \le C\left(
    \sqrt{\frac{v\log(1/\delta)}{n}} +
    \sqrt{\frac{\left(\Tr(\Sigma)+v\right)}{n}}
+ \frac{\kappa_\beta}{n^{(\beta-1)/2}}\right)~,
\]
where $\kappa_\beta$ is a constant that depends on $\beta$ and the
$\beta$-th raw moment of $\|X\|$. Thus, if the prior parameter $v$ 
is a good estimate of $\lambdamax$ in the sense that it is bounded by
a constant multiple of it, then the first two terms of the bound
are of the desired sub-Gaussian form. The third term is of smaller
order though again, it can be arbitrarily large if the mean is far
from the origin, which may be remedied by making the estimator
more complex. We refer to Catoni and Giulini \cite{CaGi17} for other
estimates of a similar spirit and more discussion. The proof
techniques of \cite{CaGi17,CaGi18} rely on so-called
\emph{PAC-Bayesian} inequalities whose details go beyond the scope of
this survey.

\subsection{Median-of-means tournaments}
\label{sec:mmt}

Here we introduce a mean estimator with a sub-Gaussian performance for
all distributions whose covariance matrix exists, proposed by Lugosi and Mendelson \cite{LuMe16a}.
The estimator presented below is the first and simplest instance of what we call \emph{median-of-means tournaments}.

Recall that we are given an i.i.d.\ sample $X_1,\ldots,X_n$
of random vectors in $\Rd$.
As in the case of the median-of-means
estimator, we start by partitioning the set
$\{1,\dots,n\}$ into $k$ blocks $B_1,\ldots,B_k$, each of
size
$|B_j|\geq m \defeq \lfloor n/k\rfloor$, where $k$ is a parameter
of the estimator whose value depends on the desired confidence level,
as specified below.
In order to simplify the presentation, we assume that $n$ is divisible by $k$ and therefore $|B_j|=m$
for all $j=1,\ldots,k$.

Define the sample mean within each block by
\[
Z_j=\frac{1}{m}\sum_{i\in B_j}X_i~.
\]
For each $a\in \R^d$, let
\begin{equation}
\label{eq:sa}
    T_a=\left\{ x\in \R^d: \exists J\subset [k]: |J| \ge k/2 \ \text{such that for all} \
      j\in J, \  \|Z_j-x\| \le \|Z_j-a\| \right\}
\end{equation}
and define the mean estimator by
\[
\wh\mu_n \in \argmin_{a\in \Rd} \radius(T_a)~,
\]
where $\radius(T_a)=\sup_{x\in T_a} \|x-a\|$.
Thus, $\wh\mu_n$ is chosen to minimize, over all $a\in \R^d$,
the radius of the set $T_a$ defined as the set of points $x\in \R^d$ for which
$\|Z_j-x\| \le \|Z_j-a\|$ for the majority of the blocks. If there are several minimizers, one may pick any one of them.

The set $T_a$ may be seen as the set of points in $\Rd$ that are at least as close to the point cloud 
$\{Z_1,\ldots,Z_k\}$ as the point $a$. The estimator $\wh\mu_n$ is obtained by minimizing the radius of $T_a$.

Note that the minimum is always achieved. This follows from the fact
that $\radius(T_a)$ is a continuous function of $a$
(since, for each $a$, $T_a$ is the intersection of a finite union of
closed balls, and the centers and radii of the closed balls are continuous in $a$).

One may interpret $\argmin_{a\in \Rd} \radius(T_a)$ as yet another multivariate notion
of the median of $Z_1,\ldots,Z_k$. Indeed, when $d=1$, it is a particular
choice of the median and
the estimator coincides
with the median-of-means estimator.

The following performance bound shows that the estimator has the desired sub-Gaussian performance.

\begin{theorem} \label{thm:mmtournament}
(Lugosi and Mendelson \cite{LuMe16a}.)
Let $\delta \in (0,1)$ and consider the mean
estimator $\wh{\mu}_n$ with parameter $k= \lceil 200\log(2/\delta)\rceil $.
If $X_1,\ldots,X_n$ are i.i.d.\ random vectors in $\Rd$ with mean
$\mu\in \Rd$ and covariance matrix $\Sigma$, then for all $n$, with probability at least $1-\delta$,
\[
   \left\|\wh{\mu}_n-\mu\right\|
\le  \max\left(960 \sqrt{\frac{\Tr(\Sigma)}{n}},
  240 \sqrt{\frac{\lambdamax \log(2/\delta)}{n}} \right)~.
\]
\end{theorem}

Just like the performance bound of Proposition \ref{prop:mmm},
Theorem \ref{thm:mmtournament} is ``infinite-dimensional'' in the sense that
the bound does not depend on the dimension $d$ explicitly. Indeed,
the same estimator may be defined for Hilbert-space valued random vectors
and Theorem \ref{thm:mmtournament} remains valid as long as $\Tr(\Sigma)=\EXP\|X-\mu\|^2$
is finite.

Theorem \ref{thm:mmtournament} is an outcome of the
following observation.

\begin{theorem} \label{thm:geometry}
Using the same notation as above and setting
\[
r=\max\left(960 \sqrt{\frac{\Tr(\Sigma)}{n}},
  240 \sqrt{\frac{\lambdamax \log(2/\delta)}{n}} \right)~,
\]
with probability at least $1-\delta$, for any $a \in \Rd$ such that $\|a-\mu\| \geq r$, one has $\|Z_j-a\| > \|Z_j-\mu\|$ for more than $k/2$ indices $j$.
In other words, $\|a-\mu\| \geq r$ implies that $a\notin T_{\mu}$.
\end{theorem}
Theorem \ref{thm:geometry} implies that for a `typical' collection $X_1,\ldots,X_n$, $\mu$ is closer to a majority of the $Z_j$'s when compared to any $a \in \Rd$ that is sufficiently far from $\mu$.
Obviously, for an arbitrary collection $x_1,\ldots,x_n \subset \Rd$
such a point need not even exist, and it is
 surprising that for a typical i.i.d.\ configuration, this property is satisfied by $\mu$.

The fact that Theorem \ref{thm:geometry} implies Theorem \ref{thm:mmtournament} is straightforward. Indeed,
the definition of $\wh{\mu}_n$ and
Theorem \ref{thm:geometry} imply that, with probability at least $1-\delta$,
$\radius(T_{\wh{\mu}_n}) \le \radius(T_\mu) \leq r$. Since either $\mu\in T_{\wh{\mu}_n}$ or $\wh{\mu}\in T_\mu$,
we must have $\|\wh{\mu}_n-\mu\| \leq r$, as required.

The constants appearing in Theorem \ref{thm:mmtournament} are certainly not optimal. They were obtained with the goal of making the proof transparent.

The proof of Theorem \ref{thm:geometry} is based on the following idea.
The mean $\mu$ is the minimizer of the function $f(x)= \EXP \|X-x\|^2$.
A possible approach is to use the available data to guess, for
any pair $a,b\in\R^d$, whether $f(a)< f(b)$. A natural choice is to
use a median of means estimator to decide which of the two is
better. The ``tournament" is simply a way of comparing every such
pair, as described next.
\footnote{As we explain in what follows, it suffices to ensure that the comparison is correct between $\mu$ and any point that is not too close to $\mu$.}.

To define the tournament, recall that $[n]$ is partitioned into $k$ disjoint blocks $B_1,\ldots,B_k$
of size $m=n/k$. For $a,b \in \R^d$, we say that
$a$ \emph{defeats} $b$ if
\begin{equation} \label{eq:l-2-tournament}
\frac{1}{m} \sum_{i \in B_j} \left(\|X_i-b\|^2 - \|X_i-a\|^2\right) > 0
\end{equation}
on more than $k/2$  blocks $B_j$. The main technical lemma is the following.

\begin{lemma} \label{lem:tournament}
Let $\delta\in (0,1)$, $k= \lceil 200\log(2/\delta)\rceil$, and define
\[
r=\max\left(960 \sqrt{\frac{\Tr(\Sigma)}{n}},
  240 \sqrt{\frac{\lambdamax \log(2/\delta)}{n}} \right)~.
 \]
 With probability at least $1-\delta$,
$\mu$ defeats all $b\in \Rd$ such that
$\|b-\mu\| \geq r$.
\end{lemma}

The outcome of Lemma \ref{lem:tournament} stands to reason: if $\|b - \mu\|$ is large enough, that will be reflected in `typical values' of $(\|X_i-\mu\|)_{i=1}^n$ and $(\|X_i-b\|)_{i=1}^n$. Comparing the values via \eqref{eq:l-2-tournament} ensures `stability', and the fact that $b$ is far from $\mu$ is exhibited with high probability. We stress that the probability estimate has to be \emph{uniform} in $b$. Such uniform estimates are a recurring theme in what follows.

\begin{proof}
Note that
\[
\|X_i-b\|^2 - \|X_i-\mu\|^2 = \|X_i-\mu+\mu-b\|^2 - \|X_i-\mu\|^2 = -2\inr{X_i-\mu,b-\mu}+\|b-\mu\|^2~,
\]
set $\ol{X}=X-\mu$ and put $v=b-\mu$. Thus, for a fixed $b$ that satisfies $\|b-\mu\| \geq r$, $\mu$ defeats $b$ if
\[
-\frac{2}{m}\sum_{i \in B_j} \inr{\ol{X}_i,v}+\|v\|^2>0
\]
on the majority of blocks $B_j$.

Therefore, to prove our claim we need that, with probability at least $1-\delta$, for every $v \in \R^d$ with $\|v\| \geq r$,
\begin{equation} \label{eq:basic}
-\frac{2}{m}\sum_{i \in B_j} \inr{\ol{X}_i,v}+\|v\|^2>0
\end{equation}
for more than $k/2$ blocks $B_j$.
Clearly, it suffices to show that \eqref{eq:basic} holds when $\|v\|=r$.

Consider a fixed $v \in \R^d$ with $\|v\|=r$. By Chebyshev's inequality, with probability at least $9/10$,
\[
\left|\frac{1}{m}\sum_{i \in B_j} \inr{\ol{X}_i,v}\right| \leq \sqrt{10} \sqrt{\frac{\EXP \inr{\ol{X},v}^2}{m}} \leq \sqrt{10} \|v\|\sqrt{\frac{\lambdamax}{m}}~,
\]
where recall that $\lambdamax$ is the largest eigenvalue of the covariance matrix $\Sigma$ of $X$. Thus, if
\begin{equation} \label{eq:r-cond-1}
r=\|v\| \ge 4\sqrt{10} \sqrt{\frac{\lambdamax}{m}}
\end{equation}
then with probability at least $9/10$,
\begin{equation} \label{eq:basic-1}
-\frac{2}{m}\sum_{i \in B_j} \inr{\ol{X}_i,v} \geq \frac{-r^2}{2}.
\end{equation}
Applying Hoeffding's inequality (\cite{Hoe63}), we see that
\eqref{eq:basic-1} holds for a single $v$ with probability at least
$1-\exp(-k/50)$ on at least $8/10$ of the blocks $B_j$.

Now we need to extend the above from a fixed vector $v$ to all
vectors with norm $r$. In order to show that
\eqref{eq:basic-1} holds simultaneously for all $v\in r\cdot S^{d-1}$
on at least $7/10$ of the blocks $B_j$, we first consider a
maximal $\epsilon$-separated set $V_1 \subset r\cdot S^{d-1}$
with respect to the $L_2(X)$ norm. In other words, $V_1$ is a subset of
$r\cdot S^{d-1}$ of maximal cardinality such that for all $v_1,v_2\in V_1$,
$\|v_1-v_2\|_{L_2(X)}= \inr{v_1-v_2,\Sigma(v_1-v_2)}^{1/2} \ge \epsilon$. We may estimate this cardinality
by the ``dual Sudakov'' inequality
(see Ledoux and Talagrand \cite{LeTa91} and also Vershynin\cite{Ver09}
for a version with the specific constant used here):
the cardinality of $V_1$ is bounded by
\[
\log (|V_1|/2) \leq \frac{1}{32} \left(\frac{\EXP\left[\inr{G,\Sigma G}^{1/2}\right]}{\epsilon/r}\right)^2~,
\]
where $G$ is a standard normal vector in $\R^d$.
Notice that for any $a \in \Rd$, $\EXP_X \inr{a,X}^2 = \inr{a,\Sigma a}$, and therefore,
\begin{eqnarray*}
\EXP\left[\inr{G,\Sigma G}^{1/2}\right]
& = & \EXP_G \left[\left(\EXP_X
    \left[\inr{G,\ol{X}}^2\right] \right)^{1/2}\right]
\leq \left(\EXP_X \EXP_G \left[\inr{G,\ol{X}}^2 \right] \right)^{1/2} \\
& = & \left(\EXP\left[ \left\|\ol{X}\right\|^2\right]\right)^{1/2} = \sqrt{\Tr(\Sigma)}~.
\end{eqnarray*}
Hence, by setting
\begin{equation} \label{eq:mesh}
\epsilon = 2 r \left(\frac{1}{k}\Tr(\Sigma)\right)^{1/2}~,
\end{equation}
we have $|V_1|\le 2e^{k/100}$ and by the union bound,
with probability at least $1-2e^{-k/100}\ge 1-\delta/2$,
\eqref{eq:basic-1} holds for all $v\in V_1$
on at least $8/10$ of the blocks $B_j$.

Next we check that property \eqref{eq:basic} holds simultaneously for all $x$ with
$\|x\|=r$ on at least $7/10$ of the blocks $B_j$.

For every $x \in r \cdot S^{d-1}$, let $v_x$ be the nearest element to
$x$ in $V_1$ with respect to the $L_2(X)$ norm.
It suffices to show that, with probability at least $1-\exp(-k/200)\ge 1-\delta/2$,
\begin{equation} \label{eq:osc}
\sup_{x \in r\cdot S^{d-1}} \frac{1}{k} \sum_{j=1}^k \IND_{\{|m^{-1}\sum_{i \in B_j} \inr{\ol{X}_i,x-v_x}| \geq r^2/4\}} \leq \frac{1}{10}~.
\end{equation}
Indeed, on that event it follows that for every $x \in r\cdot S^{d-1}$, on at least $7/10$ of the blocks $B_j$, both
$$
-\frac{2}{m} \sum_{i \in B_j} \inr{\ol{X}_i,v_x} \geq \frac{-r^2}{2} \ \ \ {\rm and} \ \ \ 2\left|\frac{1}{m} \sum_{i \in B_j} \inr{\ol{X}_i,x}-\frac{1}{m} \sum_{i \in B_j} \inr{\ol{X}_i,v_x}\right| < \frac{r^2}{2}
$$
hold and hence, on those blocks, $-\frac{2}{m} \sum_{i \in B_j}
\inr{\ol{X}_i,x} +r^2>0$ as required.

It remains to prove \eqref{eq:osc}. Observe that
\[
\frac{1}{k} \sum_{j=1}^k \IND_{\{|m^{-1}\sum_{i \in B_j} \inr{\ol{X}_i,x-v_x}| \geq r^2/4\}} \leq \frac{4}{r^2} \frac{1}{k}\sum_{j=1}^k \left|\frac{1}{m}  \sum_{i \in B_j} \inr{\ol{X}_i,x-v_x} \right|~.
\]
Since $\|x-v_x\|_{L_2(X)} = \sqrt{\EXP \inr{\ol{X},x-v_x}^2} \leq \epsilon$ it follows that for every $j$
\[
\EXP \left|\frac{1}{m}  \sum_{i \in B_j} \inr{\ol{X}_i,x-v_x} \right|
\leq \sqrt{\frac{\EXP\left[\inr{\ol{X},x-v_x}^2\right]}{m}}
\leq \frac{\epsilon}{\sqrt{m}}~,
\]
and therefore,
\begin{eqnarray*}
\lefteqn{
  \EXP \sup_{x \in r\cdot S^{d-1}} \frac{1}{k} \sum_{j=1}^k \IND_{\{|m^{-1}\sum_{i \in B_j} \inr{\ol{X}_i,x-v_x}| \geq r^2/4\}} } \\
& \leq &
\frac{4}{r^2} \EXP \sup_{x \in r\cdot S^{d-1}} \frac{1}{k}\sum_{j=1}^k \left(\left|\frac{1}{m}  \sum_{i \in B_j} \inr{\ol{X}_i,x-v_x} \right| - \EXP \left|\frac{1}{m}  \sum_{i \in B_j} \inr{\ol{X}_i,x-v_x} \right|\right) +  \frac{4\epsilon}{r^2\sqrt{m}} \\
& \defeq & (A)+(B)~.
\end{eqnarray*}
To bound $(B)$, note that, by \eqref{eq:mesh},
\[
\frac{4\epsilon}{r^2\sqrt{m}} = 8 \left(\frac{\Tr(\Sigma)}{n}\right)^{1/2} \cdot \frac{1}{r} \leq \frac{1}{60}
\]
provided that
\[
r \geq 480 \left(\frac{\Tr(\Sigma)}{n}\right)^{1/2}~.
\]
We may bound $(A)$ by standard techniques of empirical processes such as
symmetrization, contraction for Rademacher averages and de-symmetrization. Indeed,
 let $\sigma_1,\ldots,\sigma_n$ be independent Rademacher random variables
(i.e., $\PROB\{\sigma_i=1\}=\PROB\{\sigma_i=-1\}=1/2$), independent of all of the $X_i$. Then
\begin{eqnarray*}
(A) & \leq &
\frac{8}{r^2} \EXP \sup_{x \in r\cdot S^{d-1}} \frac{1}{k}\sum_{j=1}^k \sigma_j \left|\frac{1}{m}  \sum_{i \in B_j} \inr{\ol{X}_i,x-v_x} \right|
\\
& & \text{(by a standard symmetrization inequality, see, e.g., \cite[Lemma 2.3.6]{vaWe96})} \\
& \le &
\frac{8}{r^2} \EXP \sup_{x \in r\cdot S^{d-1}} \left| \frac{1}{k}\sum_{j=1}^k \sigma_j \frac{1}{m}  \sum_{i \in B_j} \inr{\ol{X}_i,x-v_x} \right|
\\
& & \text{(by a contraction lemma for Rademacher averages, see \cite{LeTa91})} \\
& \le &
\frac{16}{r^2} \EXP \sup_{x \in r\cdot S^{d-1}} \left| \frac{1}{n}\sum_{i=1}^n  \inr{\ol{X}_i,x-v_x} \right|
\\
& & \text{(see again \cite[Lemma 2.3.6]{vaWe96})} \\
& \le &
\frac{32}{r} \EXP \sup_{\{t:\|t\|\le 1\}} \left|\frac{1}{n} \sum_{i=1}^n \inr{\ol{X}_i,t}\right| \\
& & \text{(noting that $\|x-v_x\|\le 2r$)} \\
& \leq &  \frac{32}{r} \cdot \frac{\sqrt{\EXP\left\|\ol{X}\right\|^2}}{\sqrt{n}}
= \frac{32}{r} \left(\frac{\Tr(\Sigma)}{n} \right)^{1/2}
\leq \frac{1}{30}
\end{eqnarray*}
provided that
$r \geq 960 \left(\frac{\Tr(\Sigma)}{n}\right)^{1/2}.$

Thus, for
\[
Y=  \sup_{x \in r\cdot S^{d-1}} \frac{1}{k} \sum_{j=1}^k \IND_{\{|m^{-1}\sum_{i \in B_j} \inr{\ol{X}_i,x-v_x}| \geq r^2/4\}}~,
\]
we have proved that $\EXP Y \le 1/60+1/30= 1/20$. Finally, in order to establish \eqref{eq:osc}, it suffices to show that,
$\PROB\{ Y > \EXP Y + 1/20\}\le e^{-k/200}$, which follows from the
bounded differences inequality
(see, e.g., \cite[Theorem 6.2]{BoLuMa13}).
\end{proof}

\subsubsection*{Proof of Theorem \ref{thm:geometry}}

Theorem \ref{thm:geometry} is easily derived from Lemma \ref{lem:tournament}.
Fix a block $B_j$, and recall that $Z_j=\frac{1}{m}\sum_{i \in B_j}X_i$.
Let $a,b \in \R^d$. Then
\begin{eqnarray*}
\frac{1}{m}\sum_{i \in B_j} \left(\|X_i-a\|^2- \|X_i -b\|^2 \right)
& = & \frac{1}{m}\sum_{i \in B_j} \left( \|X_i-b-(a-b)\|^2- \|X_i
  -b\|^2 \right)
\\
& = & -\frac{2}{m}\sum_{i \in B_j} \inr{X_i-b,a-b} + \|a-b\|^2 = (*)
\end{eqnarray*}
Observe that $-\frac{2}{m}\sum_{i \in B_j} \inr{X_i-b,a-b} = -2\inr{\frac{1}{m}\sum_{i \in B_j} X_i -b,a-b}=-2\inr{Z_j-b,a-b}$, and thus
\begin{eqnarray*}
(*) &= & -2\inr{Z_j-b,a-b} + \|a-b\|^2 \\
& = & -2\inr{Z_j-b,a-b} + \|a-b\|^2 + \|Z_j-b\|^2 - \|Z_j-b\|^2 \\
& = & \|Z_j-b - (a-b)\|^2 - \|Z_j-b\|^2 =
\|Z_j-a\|^2-\|Z_j-b\|^2~.
\end{eqnarray*}
Therefore, $(*)>0$ (i.e., $b$ defeats $a$ on block $B_j$)
if and only if $\|Z_j-a\| > \|Z_j-b\|$.

Recall that Lemma \ref{lem:tournament} states that, with probability at least $1-\delta$,
if $\|a-\mu\| \geq r$ then on more than $k/2$ blocks $B_j$, $\frac{1}{m}\sum_{i \in B_j} \left(\|X_i-a\|^2- \|X_i -\mu\|^2 \right) >0$, which, by the above argument, is the same as saying that for at least $k/2$ indices $j$,
$\|Z_j-a\| > \|Z_j-\mu\|$.
\endproof

Upon reflection it is clear that the ideas used in the proof of
Theorem \ref{thm:mmtournament} are rather general. In fact, they are
at the heart of the \emph{small-ball method} introduced in 
Mendelson \cite{Men15} (see also \cite{Men18} for results of similar flavour). The
small-ball method holds in far more general situations than Theorem
\ref{thm:mmtournament} and will be repeated throughout this note. To
explain how the argument can be extended, let us outline again the
three steps that allowed us to compare every $b$ and $\mu$:
\begin{description}
\item{$(1)$} For any \emph{fixed} $b \in \R^d$ we obtain a bound that holds with high probability;
\item{$(2)$} Then, thanks to the high probability estimate from $(1)$, we invoke the union bound and control a large (yet finite) collection of points.

    We have complete freedom to choose the collection as we want, and we select it as an $\epsilon$-net in the set in question.
\item{$(3)$} The crucial part of the argument is passing from the control we have on every point in the net to the wanted uniform control on entire class; specifically, we show that if a `center', that is, an element of the net, is well-behaved\footnote{in the proof of Theorem \ref{thm:mmtournament}, `well-behaved' means that \eqref{eq:basic} holds for a majority of the blocks.}, then the same is true for any point close enough to the center. To that end, we show that `random oscillations' do not destroy the good behaviour of a center on too many blocks.
\end{description}


\subsection{Computational considerations}

An important issue that we have mostly swept under the rug so far is
computational feasibility of mean estimators. While the empirical mean 
is trivial to compute, many of the more sophisticated estimators
discussed here are far from being so. In particular, a basic
requirement for any multivariate mean estimator for having a chance to
being useful in practice is that it can be computed in polynomial time
(i.e., in time that is a polynomial of the sample size $n$ and the
dimension $d$). As we already pointed it out, some of the estimators described above fall in this 
category. For example, the geometric median-of-means estimator or the
Catoni-Giulini estimator are both efficiently computable in this
sense. However, these estimators fall short from being sub-Gaussian. 
The median-of-means tournament estimator is sub-Gaussian but
its computation poses a highly nontrivial challenge. In fact, the way 
the estimator is defined, it is likely to be computationally
intractable (i.e., NP hard). However, in a recent beautiful paper,
Hopkins \cite{Hop18} defines a clever semi-definite relaxation of the
median-of-means tournament estimator that can be computed in time
$O(nd + d\log(1/\delta)^c)$ for a dimension-independent constant and, at the same time, achieves the 
desired sub-Gaussian guarantee under the only assumption that the
covariance matrix exists. This is the first efficiently computable
sub-Gaussian multivariate mean estimator. 
Even more recently, Cherapanamjeri, Flammarion, and Bartlett \cite{ChFlBa19}
improved the running time to $O(nd + d\log(1/\delta)^2 + \log(1/\delta)^4)$
by combining Hopkins' ideas with clever gradient-descent optimization.
This is likely not the last
word on the subject as many exciting computational challenges arise 
in the context of mean estimation and regression. 

In the theoretical computer science community there has been a recent 
important surge of results that address the problem of
computationally efficient \emph{robust} mean estimation. In this
context, an estimator is defined to be robust if it performs well 
in the presence of a small constant fraction of  (possibly
adversarial) outliers. Various different models have been introduced,
see
Charikar, Steinhardt, and Valiant \cite{ChStVa17},
Diakonikolas, Kamath, Kane, Li, Moitra, and Stewart
\cite{DiKaKaLiMoSt16,DiKaKaLiMoSt17,DiKaKaLiMoSt18},
Diakonikolas, Kane, and Stewart \cite{DiKaSt16},
Diakonikolas, Kong, and Stewart \cite{DiKoSt18},
Hopkins and Li \cite{HoLi18},
Klivans, Kothari, and Meka \cite{KlKoMe18},
Kothari, Steinhardt, and Steurer \cite{KoStSt18},
Lai, Rao, and Vempala \cite{LaRaVe16},
Loh and Tan \cite{LoTa18},
for a sample of this important growing body of literature.
Surveying this area goes beyond the scope of this paper.

\section{Uniform median-of-means estimators}
\label{sec:gennorm}

The median-of-means tournament used in the previous section is an
example of a uniform median-of-means estimator: given a class of
functions $\F$, there is a high-probability event on which, for every
$f$ in the class, the median of means estimator based on the data $f(X_1),\ldots,f(X_n)$ is close to the mean
$\E f(X)$.  Indeed, the tournament is simply a median-of-means estimator
that was used to check whether $a$ was closer to $\mu$ than $b$, or
vice-versa, uniformly for every pair $a,b \in \R^n$.

In what follows we present a general version of a uniform
median-of-means estimator and turn our attention to two applications:
estimating the mean of a random vector with respect to an arbitrary
norm, and $L_2$-distance oracles (the latter proves useful
in regression problems, see Section \ref{sec:distance-oracle} and
\cite{LuMe16} for more details).

Formally, the question we consider is as follows:

\begin{tcolorbox}
Let $\F$ be a class of functions on a probability space
$(\Omega,\nu)$. Given an independent  sample $(X_1,\ldots,X_n)$
distributed according to $\nu$, find an estimator $\wh{\phi}_n(f)$ for
each $f\in \F$, such that, with high probability, for every $f \in \F$, $|\wh{\phi}_n(f) - \E f(X) |$ is small.
\end{tcolorbox}

A natural idea is to define $\wh{\phi}_n(f)$ to be the median-of-means
estimator based on $f(X_1),\ldots,f(X_n)$.
It stands to reason that the bound established in Section
\ref{sec:MOM} for the performance of the median-of-means estimator
cannot simply hold uniformly for every $f \in \F$.
Rather, the uniform error consists of two terms: the `worst'
individual estimate for a function $f \in \F$, and a `global' error,
taking into account the `complexity' of the class.

To analyze uniform median-of-means estimators, it is natural to
follow the path of the small-ball method outlined in the previous
section. 
To this end, fix integers $k$ and $m$ and let
$n=mk$. As always, we split the given sample into $k$ blocks, each one
of cardinality $m$, keeping in mind that the natural choice is
$k \sim \log(2/\delta)$ if one wishes a confidence of $1-\delta$. For
$0<\eta<1$ set
$$
p_m(\eta) = \sup_{f \in \F} \PROB\left(\left|\frac{1}{m}\sum_{i=1}^m f(X_i)-\E f(X) \right|\geq \eta \right)~,
$$
denote by $D=\left\{f: \EXP f(X)^2\le 1\right\}$ the unit ball in $L_2(\nu)$ and let ${\cal M}(\F,rD)$ be the maximal cardinality of a subset of $\F$ that is $r$-separated with respect to the $L_2(\nu)$ norm.

The following bound was recently established in \cite{LuMe18a}.

\begin{theorem} \label{thm:uniform-MOM}
There exist absolute constants $c_0,\ldots,c_4$ for which the following holds. Set $\eta_0,\eta_1$ and $\eta_2 \geq c_0 \eta_1/\sqrt{m}$ that satisfy the following:
\begin{description}
\item{$(1)$} $p_m(\eta_0) \leq 0.05$~;
\item{$(2)$} $\log {\cal M}(\F,\eta_1 D) \leq c_2 k \log(e/p_m(\eta_0))$~; 
\item{$(3)$} $\E \sup_{w \in \ol{W}} \left|\sum_{i=1}^n \eps_i w(X_i) \right| \leq c_3 \eta_2 n$~,
\end{description}
where $\eps_1,\ldots,\eps_n$ are independent Rademacher random
variables (i.e., $\PROB\{\eps_i=1\}=\PROB\{\eps_i=-1\}=1/2$) and 
 $W=(\F-\F) \cap \eta_1 D = \{f_1-f_2 : f_1,f_2 \in \F, \ \|f_1-f_2\|_{L_2} \leq \eta_1\}$ and $\ol{W}=\{w -\E w : w \in W\}$.
Let $r=\eta_0+\eta_2$. Then, with probability at least $1-2\exp(-c_4k)$, for all $f \in \F$ one has
$$
\left| \frac{1}{m} \sum_{i \in B_j} f(X_i) - \E f  \right| \leq r \ \ {\rm for \ at \ least \ } 0.6k \ {\rm blocks \ } B_j~.
$$
\end{theorem}

The error $r$ in Theorem \ref{thm:uniform-MOM} has the two terms we
expected. 
Indeed, $\eta_0$ is error one would have if the goal were to obtain an
individual mean estimator for a fixed function in $\F$: 
writing $\sigma_f= \sqrt{\var(f(X))}$, by Chebyshev's inequality, for every $f \in \F$,
$$
\PROB\left(\left|\frac{1}{m}\sum_{i=1}^m f(X_i)-\E f \right|\geq \eta_0 \right) \leq \frac{\sigma_f^2}{m \eta_0^2} \leq 0.05
$$
provided that
$$
\eta_0 \gtrsim \frac{\sigma_f}{\sqrt{m}} \sim \sigma_f \sqrt{\frac{\log(2/\delta)}{n}}~.
$$
As outlined in Section \ref{sec:MOM}, this leads to the standard sub-Gaussian error estimate for the function $f \in \F$.
On the other hand, $\eta_2$ involves the Rademacher averages
associated with $\F-\F$, and captures the price one has to pay for the
uniform control over the class $\F$.

The proof of Theorem \ref{thm:uniform-MOM} follows the same path we
outlined previously: the definition of $p_m$ allows us to show that
the empirical mean of $f$ on a block $B_j$ of cardinality $m$ is close
to the true mean with reasonable probability, say, larger than
$0.95$. Thus, with probability $1-e^{-ck}$, this
property is satisfied on $0.9k$ blocks. Next, the high-probability
estimate combined with the union bound allow us to control all the
elements in a finite set uniformly, as long as its cardinality is at
most exponential in $k$. The set of choice is an appropriate net in
$\F$ and its mesh width $\eta_1$ is selected to ensure that the cardinality
of the net is small enough. Finally, as always, the crucial component
is to ensure that oscillations do not `corrupt' the good behaviour on 
too many blocks. Since our interest is in the median of means, one can
live with up to $0.4k$ of the blocks being corrupted, and the
additional error of $\eta_2$ suffices to guarantee that indeed no more
than $0.4k$ blocks are affected.

The technical analysis can be found in \cite{LuMe18a}, 
where Theorem \ref{thm:uniform-MOM} is used for the study the problem
of multivariate mean estimation with respect to a general norm,
outlined in the next section.

We mention here that uniform estimators based on Catoni's mean estimator
were studied by Brownlees, Joly, and Lugosi \cite{BrJoLu15} in the
context of regression function estimation.
Minsker \cite{Min18b} discusses uniform estimators in a similar spirit
to those presented here, also for adversarially contaminated data.

\subsection{Multivariate mean estimation---the general case} 
\label{sec:vector-MOM-general}

To illustrate the power of the uniform median-of-means bounds
established in the previous section, we now return to the problem of 
estimating the mean of a random vector. As before, let
$X_1,\ldots,X_n$ be independent, identically distributed random
vectors in $\R^d$ with mean $\mu$ and covariance matrix $\Sigma$.
The question we seek to answer is to what extent one can estimate 
$\mu$ when the error is measured by a given norm $\|\cdot\|$ that is
not necessarily the Euclidean norm. 
An important example is the matrix operator norm, see  Minsker
\cite{Min18}, Catoni and Giulini \cite{CaGi17},
Mendelson and Zhivotovskiy \cite{MeZh18}.

One may now cast this general mean estimation problem in the framework of
uniform median-of-means estimators outlined above. 
The natural class of functions associated with the problem is the
unit ball with respect to the dual of the norm $\|\cdot\|$ (i.e., the set of norm-one linear functionals). 
The natural choice of a measure $\nu$ is the one induced by $X-\mu$.

Consider the event given by Theorem \ref{thm:uniform-MOM} for this
class of functions and denote the resulting error by $r$. It follows
that for each norm-one functional $x^*$, we have $\E x^*(X-\mu)=0$ and
$$
\left|\frac{1}{m} \sum_{i \in B_j} x^*(X_j-\mu)\right| \leq r
$$
for a majority of the  blocks $B_j$. Moreover,
$$
\frac{1}{m} \sum_{i \in B_j} x^*(X_i - \mu) = x^* \bigl(\frac{1}{m}\sum_{i \in B_j} X_i\bigr)- x^*(\mu)~.
$$
Thus, setting $Z_j = \frac{1}{m}\sum_{i \in B_j} X_j$, Theorem \ref{thm:uniform-MOM} implies that for every norm-one functional $x^*$,
$$
|x^*(Z_j) - x^*(\mu)| \leq r~.
$$
In other words, if we define the sets
$$
\mathbbm{S}_{x^*} = \left\{ y \in \R^d : |x^*(Z_j) - x^*(y) | \leq r \ \ {\rm for \ the \ majority \ of \ indices \ } j \right\}
$$
then on the event from Theorem \ref{thm:uniform-MOM} one has that
$$
\mu \in \mathbbm{S}= \bigcap_{\|x^*\|=1} \mathbbm{S}_{x^*}~.
$$
From a geometric point of view, each set $\mathbbm{S}_{x^*}$ is the union of intersection of slabs: setting $\alpha_j = x^*(Z_j)$,
$$
\mathbbm{S}_{x^*}= \bigcup_{|I| \geq [k/2]+1} \bigcap_{i \in I} \{ y : |x^*(y) - \alpha_j| \leq r\}~,
$$
which is just a union of (potentially empty) slabs, defined by the
functional $x^*$. The set $\mathbbm{S}$ is the resulting intersection
of the sets $\mathbbm{S}_{x^*}$. Off hand, there is no reason why the
intersection of the sets $S_{x^*}$ should not be empty. The fact that it
contains $\mu$ is only due to the special nature of the $Z_j$'s.

Since each set $\mathbbm{S}_{x^*}$ is data-dependent, so is
$\mathbbm{S}$. With that in mind, the estimator we propose is obvious:
set $\wh{\mu}_n^{(r)}$ to be any point in $\mathbbm{S}$. To show that
$\|\wh{\mu}_n^{(r)}-\mu\| \leq 2r$, fix any norm-one functional
$x^*$. Recall that if $ y \in \mathbbm{S}$ then
$|x^*(Z_j)-x^*(y)| \leq r$ on the majority of blocks. Therefore, if
$\mu,\wh{\mu}_n^{(r)} \in \mathbbm{S}$ there is some index $j$ such that,
simultaneously,
$$
|x^*(Z_j)-x^*(\wh{\mu}_n^{(r)})| \leq r \ \ \ {\rm and} \ \ \ |x^*(Z_j)-x^*(\mu)| \leq r~,
$$
and therefore $|x^*(\wh{\mu}^{(r)}-\mu)| \leq 2r$. Thanks to Theorem \ref{thm:uniform-MOM}, there is a high-probability event on which this is true for any norm-one functional, and, in particular,
$$
\|\wh{\mu}_n^{(r)} - \mu\| = \sup_{\|x^*\|=1} |x^*(\wh{\mu}_n^{(r)}-\mu)| \leq 2r~.
$$

\begin{remark}
  It is straightforward to verify that there is no need to select $\F$
  to be the set of all the norm-one linear functionals. It is enough to define
  $\mathbbm{S}$ using the functionals that are extreme points of the
  unit ball in the dual space.
\end{remark}

Thanks to Theorem \ref{thm:uniform-MOM} and the argument we just outlined, the following was established in \cite{LuMe18a}:

\begin{theorem} \label{thm:norm-estimation}
Let $\| \cdot \|$ be a norm on $\R^d$. Suppose that the $X_i$ have
mean $\mu$ and covariance matrix $\Sigma$. There exists a mean estimator
$\wh{\mu}_n$ such that,
with probability at least $1-\delta$,
\[
\|\wh{\mu}_n - \mu\| \leq \frac{c}{\sqrt{n}} \left( \max\left\{\E\|\zeta_n\|, \ \ \E \|G\| + R \sqrt{\log(2/\delta)} \right\}\right)~,
\]
where $c$ is a numerical constant, 
$$
R=\sup_{\|x^*\|=1} \left(\E (x^*(X-\mu))^2\right)^{1/2}~,
$$
$$
\zeta_n = \frac{1}{\sqrt{n}} \sum_{i=1}^n \eps_i(X_i -\mu)~,
$$
$\{\epsilon_i\}$ is a sequence of i.i.d.\ Rademacher random variables independent of $\{X_i\}$,
and $G$ is the centered Gaussian vector with covariance
matrix $\Sigma$. 
\end{theorem}

As is explained in \cite{LuMe18a}, Theorem \ref{thm:norm-estimation}, 
there is a good reason to believe that the bound of the theorem is
optimal in a rather strong sense. 
We refer the reader to \cite{LuMe18a} for more details.

\begin{remark}
  Note that the error in Theorem \ref{thm:norm-estimation} has two
  types of terms: $\frac{R}{\sqrt{n}} \sqrt{\log(2/\delta)}$ is the
  standard one-dimensional sub-Gaussian error, and its source is
  the marginal $x^*(X)$ with the largest variance. At the same time,
  $\E\|G\|$ and $\E\|\zeta_n\|$ are `global' parameters that calibrate the
  `complexity' of the norm. When $\|\cdot\|$ is the Euclidean norm, 
we recover the two terms on the right-hand side of (\ref{eq:subgaussmulti}).
\end{remark}

\subsection{$L_2$ distance oracles} \label{sec:distance-oracle}

In this section we sketch how the ideas used in Theorem \ref{thm:uniform-MOM} may be used to generate a median-of-means based (isomorphic) \emph{$L_2$ distance oracle}. A more accurate description of distance oracles may be found in \cite{LuMe16}.

Suppose $\F$ is a class of real-valued functions on a probability
space $(\Omega,\nu)$ and let $X$ be a random variable distributed as
$\nu$.  There are many natural situations in which, given an i.i.d.\
sample $X_1,\ldots,X_n$, one would like to have an accurate estimate
on the $L_2$ distance $\|f-h\|_{L_2}\defeq \sqrt{\EXP (f(X)-h(X))^2}$
between any two class members $f,h\in \F$.

In some cases, the estimates are required to be almost isometric,
that is, with high probability, for all $f,h\in \F$, the estimate should lie in the range
$[(1-\eta)\|f-h\|_{L_2}, (1+\eta)\|f-h\|_{L_2}]$ for some small value
of $\eta$. However, in many
situations (for example, in the regression problem we describe in
Section \ref{sec:regression}),
 a weaker property suffices: one would like to define a data-dependent
functional $\wh{\Psi}_n$ such that , with 'high' probability, for
all $f,h \in \F$ and a `small' value $r$, and some constants $0<\alpha<1<\beta$,
\begin{description}
\item{$\bullet$} if $\wh{\Psi}_n(f,h) \geq r$ then $\alpha \|f-h\|_{L_2}\leq \wh{\Psi}_n(f,h) \leq \beta \|f-h\|_{L_2}$;
\item{$\bullet$} if $\wh{\Psi}_n(f,h) \leq r$ then $\|f-h\|_{L_2} \leq r/\alpha$.
\end{description}
In other words, for every $f,h \in \F$, based on the value of the
data-dependent functional $\wh{\Psi}_n(f,h)$ one may estimate
$\|f-h\|_{L_2}$ in an isomorphic way---i.e., up to multiplicative constants.
We call such a functional a \emph{distance oracle}.

For the sake of simplicity, instead of considering simultaneous
estimation of pairwise distances of functions, we address the problem
of estimating $L_2$ norms of functions. In other words, given a class $\F$
of functions as above, we are interested in constructing a 
data-dependent functional $\wh{\Psi}_n$ such that if $\wh{\Psi}_n(f)
\geq r$ then  $\alpha \|f\|_{L_2} \leq \wh{\Psi}_n(f) \leq \beta
\|f\|_{L_2}$, and if $\wh{\Psi}_n(f) < r$ then $\|f\|_{L_2}
\lesssim r/\alpha$. 
Such a functional may be called a \emph{norm oracle}.
Given a norm oracle, one may construct a distance oracle in an obvious
way.

In what follows we assume that there is some $q>2$ such that the $L_q$
and $L_2$ norms are equivalent on
$\{f_1-f_2 : f_1,f_2 \in \F \cup \{0\} \}$. In other words, there is a
constant $L$ such that $\|f_1-f_2\|_{L_q} \leq L \|f_1-f_2\|_{L_2}$
for all
$f_1,f_2 \in \F \cup \{0\}$.  Consider the set
$$
H={\rm star}(\F,0) = \{\lambda f : f \in \F, \ 0 \leq \lambda \leq 1\}
$$
and let $H_\rho = H \cap \rho S(L_2)$, where $\rho S(L_2) = \{h : \|h\|_{L_2} =\rho\}$. For every $h \in H$, set
$$
Z_h(j) = \frac{1}{m} \sum_{i \in B_j} |h(X_i)|
$$
and our estimator $\wh{\Psi}_n(h)$ is the median of $Z_h(1),\ldots,Z_h(k)$.

Recall that $D=\left\{f: \EXP f(X)^2\le 1\right\}$ denotes the unit ball in $L_2(\nu)$ and let ${\cal M}(\F,rD)$ be the maximal cardinality of a subset of $\F$ that is $r$-separated with respect to the $L_2(\nu)$ norm.

\begin{theorem} \label{thm:distance}
There exist constants $c_1, A, B$ that depend on $q$ and $L$, and absolute constants $c_2,\ldots,c_6$ such that the following holds. Let $m=c_1(L,q)$ and set $k=n/m$. Under the $L_q-L_2$ norm equivalence condition, if
$$
\log{\cal M}(H_\rho, c_2 A \rho D) \leq c_3k~,
$$
and
$$
\E \sup_{w \in (H_\rho-H_\rho) \cap c_2 A \rho D} \left|\sum_{i=1}^n \eps_i w(X_i) \right| \leq c_4 A \rho n~,
$$
then with probability at least $1-2\exp(-c_5k)$, for all $h\in H_\rho$, 
\begin{description}
\item{$\bullet$} if $\|h\|_{L_2} \geq \rho$ then $A \|h\|_{L_2} \leq \wh{\Psi}_n(h) \leq B \|h\|_2$; and
\item{$\bullet$} if $\|h\|_{L_2} \leq \rho$ then $\wh{\Psi}_n(h) \leq c_6 B \rho$.
\end{description}
\end{theorem}

Note that Theorem \ref{thm:distance} shows that $\wh{\Psi}_n$ is a desired norm oracle: if $\wh{\Psi}_n(h) > c_6 B \rho$ then it follows that $\|h\|_{L_2} \geq \rho$, and thus
$$
B^{-1} \wh{\Psi}_n(h) \leq \|h\|_{L_2} \leq A^{-1} \wh{\Psi}_n(h)~.
$$
On the other hand, if $\wh{\Psi}_n(h) \leq c_6 B \rho$ then one has
two options: either $\|h\|_{L_2} \leq \rho$, or,
$\|h\|_{L_2} \geq \rho$, in which case
$\|h\|_{L_2} \leq A^{-1} \wh{\Psi}_n(h) \leq c_6(B/A) \rho$. Thus,
$\|h\|_{L_2} \leq \rho \max\{1,c_6B/A\}$. The norm oracle is obtained
 by setting $r = c_6 B \rho$ and choosing $\alpha$ and $\beta$
appropriately.

The proof of Theorem \ref{thm:distance} follows the small-ball method:
we begin by showing that for a fixed $h \in H_\rho$, and with high
probability,
\begin{equation} \label{eq:cond-net-distances}
\left|\left\{ j : A \rho \leq \frac{1}{m} \sum_{i \in B_j} |h|(X_i) \leq B \rho \right\} \right| \geq 0.8k
\end{equation}
for some constants $A$ and $B$.

Then, the high-probability estimate allows us to control a satisfactory net in $H_\rho$, and finally,
one has to control `oscillations': a high-probability event such that if $h \in H_\rho$ and $\pi h$ denotes the closest point to $h$ in the net, then
$$
\sup_{h \in H_\rho} \left| \left\{ j : \frac{1}{m}\sum_{i \in B_j} |h-\pi h|(X_i) \geq \frac{A \rho}{2}\right\} \right| \leq 0.2k~.
$$
With all three components in place, it is evident that for every $h \in H_\rho$ there are at least $0.6k$  blocks $B_j$ on which
\begin{equation*}
A \rho \leq \frac{1}{m} \sum_{i \in B_j} |\pi h|(X_i) \leq B \rho \ \ \ {\rm and} \ \ \ \frac{1}{m}\sum_{i \in B_j} |h-\pi h|(X_i) \leq \frac{A \rho}{2}~.
\end{equation*}
On these  blocks,
$$
\frac{1}{m} \sum_{i \in B_j} |h|(X_i) \geq \frac{1}{m} \sum_{i \in B_j} |\pi h|(X_i) - \frac{1}{m}\sum_{i \in B_j} |h-\pi h|(X_i) \geq \frac{A \rho}{2}~,
$$
and a similar estimate holds for the upper bound.

Once an isomorphic estimate is established in
$H_\rho={\rm star}(\F,0) \cap \rho S(L_2)$, the same estimate holds
for any $h \in H_r$ and any $r \geq \rho$. This is evident from the
fact that $H={\rm star}(\F,0)$ is star-shaped around $0$, implying
that every $h \in H_r$ has a `scaled down' version in $H_\rho$. In
particular, on the same event we have that if $f \in \F$ and
$\|f\|_{L_2} \geq \rho$, then
$$
\frac{A}{2} \|f\|_{L_2} \leq \wh{\Psi}_n(f) \leq 2B \|f\|_{L_2}~.
$$
The second part of the claim follows the same lines (see \cite{LuMe16} for more details).

The main question is how to ensure that \eqref{eq:cond-net-distances} holds with high enough probability. As it happens, \eqref{eq:cond-net-distances} can be verified under minimal assumptions, as we now explain.

Assume, for example, that the given class $\F$ satisfies a small-ball condition, namely, for every $\epsilon>0$ there is a constant $\kappa(\epsilon)$ such that for every $f \in \F$,
$$
\PROB(|f(X)| \leq \kappa(\epsilon)\|f\|_{L_2} ) \leq \epsilon~.
$$
Set $\epsilon=0.05$ and let $\kappa=\kappa(0.05)$. Then with probability at least $1-2\exp(-cn)$ there are at least $0.9n$ indices $i \in \{1,\ldots,n\}$ such that $|f(X_i)| \geq \kappa \|f\|_{L_2}$. At the same time,
$$
\PROB(|f(X)| \geq 10 \|f\|_{L_2}) \leq \frac{1}{100}~,
$$
implying that with probability at least $1-2\exp(-cn)$, for at least $0.9n$ indices $1 \leq i \leq n$, $|f(X_i)| \leq 10 \|f\|_{L_2}$. Thus, intersecting the two events \eqref{eq:cond-net-distances} is established with probability at least $1-2\exp(-c^\prime n)$ for $m=1$, $A =\kappa$ and $B=10$.

Of course it is true that not every random variable $f(X)$ satisfies
the small-ball condition we use above. However, there is an additional
degree of freedom that has not been exploited yet: that the random
variables one truly cares about are of the form
$Z_f(j)=\frac{1}{m}\sum_{i \in B_j} |f(X_i)|$, leaving us some room to
generate the necessary regularity. Indeed, it is straightforward to
verify that under minimal assumptions and for a small value of $m$,
the $Z_f(j)$ do satisfy a sufficiently strong small-ball condition. This
is an outcome of a Berry-Esseen type argument
\footnote{The case $q=3$
  is the standard Berry-Esseen theorem while for $2<q<3$ one may use
 generalized Berry-Esseen bounds, see \cite{Pet95}.}: if there is some $q>2$ such that
$\|f\|_{L_q} \leq L \|f\|_{L_2}$ then for $m=c(q,L)$,
$\sqrt{m}(Z_f(j)-\E|f|)$ is `close enough' to a Gaussian variable and
it follows that
$$
\PROB( |Z_f(j)| \leq c_1\|f\|_{L_2}) \leq 0.05~.
$$

\section{Median-of-means tournaments in regression problems} 
\label{sec:regression}

The problem of regression function estimation essentially amounts to
estimating \emph{conditional} expectations and as such, it is a
natural candidate for extending ideas of mean estimation discussed 
in this paper. In this section we explore some of the recent progress 
made in the study of regression problems driven by 
uniform median-of-means estimators. 

The standard setup for regression function estimation may be
formulated as follows. Let $(X,Y)$ be a pair of random variables such
that $X$ takes its values in some set $\X$ while $Y$ is real valued.
Given a class $\F$ of real-valued functions defined on $\X$,
one's goal is  to find $f\in \F$ for which $f(X)$ is a good prediction
of $Y$. The performance of a predictor $f\in \F$ is measured by the 
\emph{mean-squared error} $\EXP (f(X)-Y)^2$, also known as the
\emph{risk}.
The best performance in the class is achieved by the risk minimizer
\[
   f^*= \argmin_{f\in \F} \EXP (f(X)-Y)^2~.
\]
We  assume in what follows that the minimum is attained and $f^*\in
\F$ exists and is unique.

The difficulty stems from the fact that the joint distribution of
$(X,Y)$ is not known. Instead, one is given an i.i.d.\ sample
$\cD_n=(X_i,Y_i)_{i=1}^n$ distributed according to the joint distribution
of $X$ and $Y$. 
Given a sample size $n$, a \emph{learning procedure} is a map
that assigns to each sample
$\cD_n$ a (random) function in $\F$, which we denote by
$\wh{f}$. 

The success of $\wh{f}$ is measured by the tradeoff
between the accuracy $\epsilon$ and the confidence $\delta$ with which $\wh{f}$
attains that accuracy, that is, one would like to find $\wh{f}$
which satisfies that
$$
\PROB \left( \E \left(\left(\wh{f}(X)-Y\right)^2 | \cD_n \right) \leq \inf_{f \in \F} \E (f(X)-Y)^2 + \epsilon \right) \geq 1-\delta
$$
for values of $\epsilon$ and $\delta$ as small as possible.
\footnote{Note that one has the freedom to select a function $\wh{f}$ that does not belong to $\F$.}.
The question of this accuracy/confidence
tradeoff has been the subject of extensive study, see, for example, the books
\cite{VaCh74a,DeGyLu95,GyKoKrWa02,vaWe96,AnBa99,vdG00,Mas06,Kolt08,TS09,BuvdG11}
for a sample of the large body of
work devoted to this question.

The most standard and natural way of choosing $\wh{f}$ is by \emph{least squares regression},
also known as \emph{empirical risk minimization}:
\[
  \wh{f} = \argmin_{f\in \F} \sum_{i=1}^n (f(X_i)-Y_i)^2~.
\]
A sample of the rich literature on the analysis of empirical risk minimization
includes Gy\"orfi, Kohler, Krzyzak, Walk \cite{GyKoKrWa02}, van de
Geer \cite{vdG00}, Bartlett, Bousquet, and Mendelson \cite{BaBoMe05},
Koltchinskii \cite{Kolt08}, Massart \cite{Mas06}.

The simple idea behind empirical risk minimization is that, for each
$f\in \F$, the empirical risk $(1/n) \sum_{i=1}^n (f(X_i)-Y_i)^2$ is a
good estimate of the risk $\EXP (f(X)-Y)^2$ and the minimizer of the
empirical risk should nearly match that of the ``true'' risk. 
Naturally, when the empirical risks are not reliable estimates of
their population counterparts, empirical risk minimization stands on
shaky ground. It should not come as a surprise that the performance of 
empirical risk minimization changes dramatically
according to the tail behaviour of the functions involved in the given
learning problem. One may show (see, e.g., \cite{LeMe16}) that if $\F$
is convex and the random
variables $\{f(X) : f \in \F\}$ and the target $Y$ have well-behaved
sub-Gaussian tails, empirical risk minimization performed in
$\F$ yields good results: it essentially attains the optimal
accuracy/confidence tradeoff for a certain range of the parameters.
However, the situation deteriorates considerably when
either members of $\F$ or $Y$ is
heavy-tailed in some sense. In such cases, the performance of
empirical risk minimization may be greatly improved by employing
more sophisticated mean estimation techniques.
For the analysis of least squares regression for some heavy-tailed
situations, see Han and Wellner \cite{HaWe17}.

A growing body of recent work has addressed the problem of
constructing regression function estimators that work well even when 
some of the $f(X)$ and $Y$ may be heavy tailed, see
Audibert and Catoni \cite{AuCa11},
Brownlees, Joly, and Lugosi \cite{BrJoLu15},
Catoni and Giulini \cite{CaGi17}.
Chichignoud and Lederer \cite{ChLe14},
Fan, Li, and Wang \cite{FaLiWa17},
Hsu and Sabato \cite{HsSa16},
Lecu{\'e} and Lerasle \cite{LeLe17,LeLe17a},
Lecu{\'e}, Lerasle, and Mathieu \cite{LeLeMa18},
Lerasle and Oliveira \cite{LerasleOliveira_Robust},
Lugosi and Mendelson \cite{LuMe16,LuMe18}, 
Mendelson \cite{Men17},
and
Minsker \cite{Min15}.

In this section we limit ourselves to sketching how median-of-means
tournaments may be used in regression function estimation.
Median-of-means tournaments were introduced in \cite{LuMe16} for the
study of such regression problems when $\F$ is a convex set. It was
shown that one may attain
the optimal accuracy/confidence tradeoff in prediction problems in
convex classes. Similar methods were used in \cite{LuMe18} and \cite{LeLe17a} to
study the regularization framework. In these papers
the convexity of the underlying class $\F$ played a central role in
the analysis. In fact, it is convexity that allows one to define an
optimal $\wh{f}$ that takes values in $\F$. In the general case, when
$\F$ need not be convex, selecting $\wh{f} \in \F$ can be a poor
choice (see, e.g. the discussion in \cite{Men15}), and one has to adopt a
totally different approach for naming an estimator.

An optimal choice of $\wh{f}$ for an arbitrary class $\F$ was
introduced by Mendelson \cite{Men17}, and that choice is also based on median-of-means tournament, though a different tournament than the one defined in \cite{LuMe16}.

Finally, we mention the general framework of $\rho$-estimators introduced by Baraud, Birg{\'e}, and Sart \cite{BaBiSa17}
and Baraud and Birg{\'e} \cite{BaBi18}. The construction of their estimators bears certain similarities with the
tournament procedures described here.

For the sake of simplicity, we will only consider the problem of
regression in a closed and convex class $\F$. We set
$$
f^*={\rm argmin}_{f \in \F} \E (f(X)-Y)^2
$$
to be the minimizer in $\F$ of the risk, and since $\F$ is convex and closed, such a minimizer exists and is unique. The excess risk of $f \in \F$ is defined to be
$$
\E {\cal L}_f = \E (f(X)-Y)^2-\E (f^*(X)-Y)^2
$$
and the aim is to ensure that $\E ({\cal L}_{\wh{f}}|D) \leq \epsilon$ with probability at least $1-\delta$.

As one may expect from a median-of-means estimator, we select
$k \leq n$ wisely, split the given sample $(X_i,Y_i)_{i=1}^n$ to $k$
blocks, each of cardinality $m=n/k$, and compare the statistical
performance of every pair of functions on each block. Just as before,
the belief is that because $\E(f^*(X)-Y)^2$ is smaller than
$\E(f(X)-Y)^2$ this fact is exhibited by a median-of-means estimate, allowing us
to prefer $f^*$ over $f$. Hence, if we can find a uniform median-of-means estimator, such a comparison would lead us to a function that
has almost the same risk as $f^*$.

With that in mind, the natural choice of a ``match'' in the tournament
between two candidate functions $f$ and $h$ is counting the number of blocks on which
$\frac{1}{m}\sum_{i \in B_j} (f(X_i)-Y_i)^2$ is larger than
$\frac{1}{m}\sum_{i \in B_j} (h(X_i)-Y_i)^2$. The function that
exhibits a superior performance (i.e., has a smaller empirical mean)
on the majority of the blocks is the winner of the match.

In a perfect world, we would choose a function that won all of its
matches. However, the world is far from perfect and the outcomes of
matches between functions that are `too close' are not reliable. To
address this issue, the tournament requires an additional component: a
\emph{distance oracle}, similar to the one presented in the previous
section. Thanks to the distance oracle one may verify in a
data-dependent way when two functions are too close, and in such cases
disregard the outcome of the match between them.

Let us describe some technical facts that are at the heart of the results in \cite{LuMe16,LuMe18}.
Define the ``quadratic'' and ``multiplier'' processes
$$
\mathbbm{Q}_{f,h}(j) = \frac{1}{m}\sum_{i \in B_j} (f(X_i)-h(X_i))^2, \ \ \ \mathbbm{M}_{f,h}(j)= \frac{2}{m}\sum_{i \in B_j} (f(X_i)-h(X_i)) (h(X_i)-Y_i)
$$
and put
$$
\mathbbm{B}_{f,h}(j) \equiv \frac{1}{m}\sum_{i \in B_j}(f(X_i)-Y_i)^2 - \frac{1}{m}\sum_{i \in B_j}(h(X_i)-Y_i)^2 = \mathbbm{Q}_{f,h}(j)+\mathbbm{M}_{f,h}(j)~.
$$
Note that $\E \mathbbm{B}_{f,h}(j)=\E (f(X)-Y)^2 - \E (h(X)-Y)^2$. Therefore, at least intuitively, if $\mathbbm{B}_{f,h}(j)>0$ for a majority of indices $1 \leq j \leq k$, one would expect that $\E(f(X)-Y)^2 > \E (h(X)-Y)^2$, making $h$ a better candidate to be a risk minimizer than $f$.

When one is given a sample $(X_i,Y_i)_{i=1}^{3n}$, the choice of $\wh{f}$ is carried out as follows:
\begin{tcolorbox}
Step 1:
\begin{description}
\item{$\bullet$} Fix $r>0$, corresponding to the desired accuracy parameter $\epsilon \sim r^2$.
\item{$\bullet$} Let $\wh{\Phi}_n$ be a distance oracle in $\F$ similar to the one described in the previous section, which uses as data the first part of the sample $(X_i)_{i=1}^n$. Thus, for the right choice of parameters $\alpha$ and $\beta$ and with high probability the following holds: if $f,h \in \F$ and $\wh{\Phi}_n(f,h) \geq \beta r$ then $\|f-h\|_{L_2} \sim_{\alpha,\beta} \wh{\Phi}_n(f,h)$, and if $\wh{\Phi}_n(f,h) \leq \beta r$ then $\|f-h\|_{L_2} \leq (\beta/\alpha)r$.

 Define ${\cal DO}(f,h)=1$ if $\wh{\Phi}_n(f,h) \geq \beta r$ and ${\cal DO}(f,h)=0$ otherwise.
\end{description}
\end{tcolorbox}
The binary valued functional ${\cal DO}$ serves as the `referee' of the tournament. Its role is to decide when a match between two functions is allowed to take place. In a more mathematical language, when ${\cal DO}(f,h)=1$ one has a good reason to expect that $f$ and $h$ are far enough to ensure that $(\mathbbm{B}_{f,h}(j))_{j=1}^k$ reflects the true value $\E(f(X)-Y)^2 - \E (h(X)-Y)^2$.

\begin{tcolorbox}
Step 2:
\begin{description}

\item{$\bullet$} This round of the tournament consists of statistical matches between class members which are preformed using the second part of the sample $(X_i,Y_i)_{i=n+1}^{2n}$. A match is allowed to proceed only if ${\cal DO}(f,h)=1$; otherwise, the match is drawn. If a match does take place then $h$ defeats $f$ if $\mathbbm{B}_{f,h}(j)>0$ for a majority of indices $j$, and $f$ defeats $h$ if the reverse inequality holds for a majority of the blocks.

\item{$\bullet$} A function $f$ qualifies from this round if it has has won or drawn all of its matches.
\end{description}
\end{tcolorbox}

The crucial fact behind Step 2 is that, with high
probability, the risk minimizer $f^*$ qualifies for the next round: if
${\cal DO}(h,f^*)=1$ then $h$ and $f^*$ are far enough to ensure that
$(\mathbbm{B}_{h,f^*}(j))_{j=1}^k$ reflects the true value
$\E(h(X)-Y)^2 - \E (f^*(X)-Y)^2$. Since $f^*$ is the unique
minimizer of the risk, the majority of values are positive.

Moreover, the same argument implies that if $h$ is a qualifier from
Step 2, then $\|h-f^*\|_{L_2} \leq \beta r$. Indeed,
the match between $h$ and $f^*$ (or between any two qualifiers) must
have been drawn by the referee's decision; thus $h$ must be `close' to $f^*$.

Step 2 is not enough to identify a function with a small excess risk. Indeed, all the qualifiers are close to $f^*$, but the fact that $\|f-f^*\|_{L_2} \leq \beta r$ does not imply that $\E(f(X)-Y)^2-\E(f^*(X)-Y)^2 \lesssim r^2$. Therefore, the tournament has an additional step: the \emph{Champions League} round, in which all the qualifiers play each other in a different type of match.

To find a function that does have an almost optimal risk one uses the third part of the sample $(X_i,Y_i)_{i=2n+1}^{3n}$ to define a `home and away' style matches:
\begin{tcolorbox}
Step 3:
\begin{description}
\item{$\bullet$} Let $\Psi_{h,f}=(h(X)-f(X))(f(X)-Y)$ and set $\Psi_{h,f}(j) = \frac{1}{m}\sum_{i \in B_j} \Psi_{h,f}(X_i,Y_i)$. Let $\alpha,\beta$ and $r$ be as above and put $r_1=2(\beta/\alpha)r$.
\item{$\bullet$}   $f$ wins its home match against $h$ if $\Psi_{h,f}(j) \geq -r_1^2/10$ for a majority of the indices $j$.
\item{$\bullet$} A winner of the tournament is any qualifier that wins all of its home matches. We set $\wh{f}$ to be any such winner.
\end{description}
\end{tcolorbox}
To see the reason behind this choice of matches, recall that all the qualifiers $h$ satisfy that $\|h-f^*\|_{L_2} \leq \beta r$. At the same time, the excess risk of $h$ is
$$
\E (h(X)-Y)^2 - \E (f^*(X)-Y)^2 = \|h-f^*\|_{L_2}^2 + 2 \E (h(X)-f^*(X)) \cdot (f^*(X)-Y)~.
$$
Since $\|h-f^*\|_{L_2}^2$ is of the order of $r^2$ it is evident that if $\E (h(X)-f^*(X)) \cdot (f^*(X)-Y) \lesssim r^2$, then the excess risk of $h$ is also of the order of $r^2$.

Observe that $\E \Psi_{h,f^*}=\E (h(X)-f^*(X)) \cdot (f^*(X)-Y)$ and that by the convexity of $\F$, $\E \Psi_{h,f^*} \geq 0$ (this follows from the characterization of the nearest point map onto a closed, convex subset of a Hilbert space). Moreover,
\begin{equation} \label{eq:reg-psi}
\E \Psi_{h,f^*}=-\|h-f^*\|_{L_2}^2 - \E\Psi_{f^*,h}~.
\end{equation}
One shows that $\Psi_{h,f^*}(j) \gtrsim -r^2$ for a majority of indices $j$. This follows because the median of $(\Psi_{h,f^*}(j))_{j=1}^k$ happens to be a uniform median-of-means estimator of the true mean $\E \Psi_{h,f^*}$. As a consequence, $f^*$ wins all of its home matches. Also, if $h$ wins a home match against $f^*$, (i.e., $\Psi_{f^*,h}(j) \gtrsim -r^2$ for a majority of indices $j$), then $\E \Psi_{f^*,h} \gtrsim - r^2$ and by \eqref{eq:reg-psi}, $\E \Psi_{h,f^*} \lesssim r^2$. That implies that every function that wins all of its home matches must have a small excess risk.

To conclude, all three components of the tournament procedure from \cite{LuMe16} are derived using uniform median-of-means estimators (of different functionals) in the class $\F$.

Without going into technical details, at the heart of the analysis of Steps $2$ and $3$ of the tournament is the following fact: given a convex class $\F$ that satisfies some minimal conditions, for the right choice of $k$ and $r$ (the choice of $r$
depends on the geometry of the class $\F$ and on the parameters
$\gamma_1$ and $\gamma_2$ appearing below), and for an absolute
constant $c_1$, we have that, with probability $1-2\exp(-c_1k)$,
\begin{tcolorbox}
\begin{description}
\item{(1)} for every $f \in \F$ such that $\|f-f^*\|_{L_2} \geq r$, one has
\[
\mathbbm{B}_{f,f^*}(j) \geq \gamma_1 \|f-f^*\|_{L_2}^2
\]
for $0.99k$ of the blocks;
\item{(2)} for every $f \in \F$ such that $\|f-f^*\|_{L_2} < r$, one has
\[
|\mathbbm{M}_{f,f^*}(j) - \E \mathbbm{M}_{f,f^*}(j)| \leq \gamma_2 r^2
\]
for $0.99k$ of the  blocks.
\end{description}
\end{tcolorbox}

These facts suffice for proving the validity of steps $(2)$ and $(3)$
in the tournament procedure. A general bound for the performance of
the procedure defined above was proven by Lugosi and Mendelson
\cite{LuMe16}. The achievable accuracy depends on the interaction
between the geometry of the class $\F$ and the distribution of
$(X,Y)$. Instead of recalling the technical details in their full generality, we simply
illustrate the performance on the canonical example of linear regression.

Let $\F=\{\inr{t,\cdot} : t \in \R^d\}$ be the class of linear
functionals on $\R^d$. Let $X$ be an isotropic random vector in $\R^d$
(i.e., $\E\inr{t,X}^2 = 1$ for every $t$ in the Euclidean unit sphere)
and assume that the distribution of $X$ is such that
there are $q>2$ and $L>1$ for which, for every
$t \in \R^d$, $\|\inr{X,t}\|_{L_q} \leq L \|\inr{X,t}\|_{L_2}$.

Assume that one is given $n$ noisy measurements of $\inr{t_0,\cdot}$
for a fixed but unknown $t_0 \in \R^d$, that is, assume that
$Y=\inr{t_0,X}+W$ for some
symmetric random variable $W$ that is independent of $X$ and has
variance $\sigma^2$.
One observes the ``noisy" data $(X_i,Y_i)_{i=1}^n$ and
the aim is to approximate $t_0$
with a small error (accuracy) and with high probability (confidence).

Invoking standard methods as in \cite{LeMe16b}, the
best that one can guarantee using empirical risk minimization is a choice of $\wh{t} \in
\R^d$, for which 
$$
\|\wh{t}-t_0\|_2 
\leq \frac{C }{\delta} \sigma\sqrt{\frac{d}{n}} \ \ \ {\rm with \ probability \ } 1-\delta-2\exp(-c_1 d)
$$
for some constant $C$ that depends on $q$ and $L$.
Therefore, if one wishes for an error that is proportional to 
$\sigma \sqrt{d/n}$, the best that one can hope for is a constant
confidence $\delta$. 

The median-of-means tournament procedure, when applied to this example, selects $\wh{t}$ for which
\[
\|\wh{t}-t_0\|_2 \leq C \sigma\sqrt{\frac{d}{n}} \ \ \ {\rm with \ probability \ } 1-2\exp(-cd)
\]
for some numerical constants $c,C>0$. As it is argued in \cite{LuMe16}
that this is the optimal confidence at any level that is proportional
to $\sqrt{d/n}$. 
In fact, the median-of-means tournament procedure gives the optimal confidence for
any accuracy $r \geq c^\prime\sigma \sqrt{d/n}$. Standard empirical
risk minimization can only achieve such accuracy/confidence tradeoff
for sub-Gaussian distributions.

\medskip
\noindent
{\bf Acknowledgements.} We thank Sam Hopkins, Stanislav Minsker, and Roberto Imbuzeiro Oliveira for
illuminating discussions on the subject.
We also thank two referees for their thorough reports and insightful comments.


\end{document}